\documentclass[11pt,a4paper,reqno]{article}

\usepackage[T1]{fontenc}
\usepackage[all]{xy}
\usepackage{epsfig}
\usepackage{amssymb}
\usepackage{amsmath}
\usepackage{bbm}
\usepackage{centernot}
\usepackage{graphics,graphicx}
\usepackage{ esint }
\usepackage{eepic}
\usepackage{epsfig}
\usepackage{times}
\usepackage{stmaryrd}

\usepackage[usenames,dvipsnames]{xcolor}

\usepackage{amsbsy}
\usepackage{amsfonts}
\usepackage{amsthm}
\usepackage{indentfirst}
\usepackage{enumerate}
\usepackage{marginnote}

\usepackage[toc,page]{appendix}
\usepackage{tikz}
\usepackage{pgfplots}

\usepgfplotslibrary{fillbetween}

\newtheorem{theorem}{Theorem}[section]
\newtheorem{lemma}[theorem]{Lemma}
\newtheorem{prop}[theorem]{Proposition}
\newtheorem{definition}[theorem]{Definition}

\newtheorem{corollary}[theorem]{Corollary}

\theoremstyle{definition}
\newtheorem{remark}[theorem]{Remark}

\def\qed{\hfill $\square$ \goodbreak \bigskip}

\def\eps{\varepsilon}
\def\R{{\mathbb R}}
\def\N{{\mathbb N}}

\def\A{{\mathcal A}}
\def\Om{\Omega}
\newcommand{\K}{\mathcal{K}}
\def\C{{\mathcal C}}

\def\H{\mathcal{H}}
\def\K{{\mathcal K}}
\def\O{{\mathcal O}}

\def\div{{\rm div}}

\def\Conv{{\rm Conv}}

\def\diam{{\rm diam}}
\def\Int{\text{Int}}
\def\inr{{\rm inr}}

\def\Id{{I\!d}}

\def\Ind{\raise 2pt\hbox{$\chi $}}

\newbox\boxint\newbox\boxtir\newdimen\longa\newdimen\longb
\def\intgm{
\setbox\boxint=\hbox{$\displaystyle\int$}
\setbox\boxtir=\hbox{$-$}
\longa=\wd\boxint\longb=\wd\boxtir
\advance\longa-\longb\divide\longa2
\advance\longb\longa
\box\boxint\hskip-\longb\box\boxtir}

\def\intm{
\setbox\boxint=\hbox{$\int$}
\setbox\boxtir=\hbox{{--}}
\longa=\wd\boxint\longb=\wd\boxtir
\advance\longa-\longb\divide\longa2
\advance\longb\longa
\box\boxint\hskip-\longb\box\boxtir}
\def\moyp{\mathop{\intm}\nolimits}
\def\moy_#1{\ifmmode\ifinner{\moyp_{#1}}\else{\fint_{#1}}\fi\fi}

\newcounter{mnotecount}[section]

\newcommand{\rmnote}[1]{}

\newcommand{\footnoteremember}[2]{\footnote{#2}\newcounter{#1}\setcounter{#1}{\value{footnote}}}
\newcommand{\footnoterecall}[1]{\footnotemark[\value{#1}]}
\title{ Regularity in shape optimization\\under convexity constraint}
\author{Jimmy {\sc Lamboley}\footnoteremember{0}{Sorbonne Universit\'e and Universit\'e Paris Cit\'e, CNRS, IMJ-PRG, F-75005 Paris, France.},  
Rapha\"el {\sc Prunier}\footnoterecall{0}}
\begin{document}

\maketitle
\begin{abstract}

This paper is concerned with the regularity of shape optimizers of a class of isoperimetric problems under convexity constraint. We prove that minimizers of the sum of the perimeter and a perturbative term, among convex shapes, are $C^{1,1}$-regular. To that end, we define a notion of quasi-minimizer fitted to the convexity context and show that any such quasi-minimizer is $C^{1,1}$-regular. The proof relies on a cutting procedure which was introduced to prove similar regularity results in the calculus of variations context. Using a penalization method we are able to treat a volume constraint, showing the same regularity in this case. We go through some examples taken from PDE theory, that is when the perturbative term is of PDE type, and prove that a large class of such examples fit into our $C^{1,1}$-regularity result. Finally we provide a counter-example showing that we cannot expect higher regularity in general. 
\smallskip

\end{abstract}

\section{Introduction}

In this paper, we study the regularity properties of minimizers in shape optimization under convexity constraint for a large class of problems of isoperimetric type. 

In the classical framework of shape optimization (classical in the sense that there is no convexity constraint), the question of regularity has a long-standing history, with strong interactions with the fields of geometric measure theory and free boundary problems. More specifically, the study of various problems involving the classical De-Giorgi perimeter $P$ leads to the notion of quasi-minimizer of the perimeter (see \eqref{eq:quasimin}) which is very useful to prove regularity for many problems of the form
\begin{equation}\label{eq:pb0}
\min\Big\{P(\Om)+R(\Om), \;\;\Om\in \A\Big\}
\end{equation}
Here $R$ is considered to be a perturbative term, and $\A$ is a given class of measurable sets, for example the class of sets of given volume $V_{0}$, or the class of sets included in a fixed box $D$, or a mix of both:
 $$\A_{1}=\{\Om\subset \R^N, \;\;|\Om|=V_{0}\},\;\;\;\;\A_{2}=\{\Om\subset \R^N,\Om\subset D\}, \;\;\;\;\A_{3}=\{\Om\subset \R^N,\ \Om\subset D, \;|\Om|=V_{0}\}$$
 though many other examples are possible (note that $|\Om|$ denotes the volume of $\Om$). It would be impossible to refer to every work in this direction, but we refer to \cite{Mag} for a nice introduction to the concept of quasi-minimizer of the perimeter and for several examples, and to \cite{SZ97Are,ACKS,landais,KM2,dPV,DLPV,P21Lar} for a short sample of applications.

In this paper, we are interested in a similar class of problems, where we add a convexity assumption to the admissible shapes. More precisely, we replace \eqref{eq:pb0} by
\begin{equation}\label{eq:pb1}
\min\Big\{P(K)+R(K), \;\;K\in \A\cap \K^N\Big\}
\end{equation}
where $\K^N$ denotes the class of convex bodies of $\R^N$ (convex compact sets with nonempty interior). As before, $P(K)$ denotes the perimeter of $K$, and as $K$ is a convex body (and therefore is a Lipschitz set) we have $P(K)=\H^{N-1}(\partial K)$ where $\H^{N-1}$ denotes the $(N-1)$-dimensional Hausdorff measure in $\R^N$. Again $R$ is a shape functional that is considered as a perturbative term, and we will make assumptions on $R$ so that the term driving the regularity of optimal shapes is the perimeter term.

Before going into more details about our motivations and our strategy, let us start by giving a consequence of the three main results of the paper that are Theorems \ref{th:shape1}, \ref{th:shape3} and \ref{th:examples}:

\begin{theorem}\label{th:intro}
Let $n\in \N^*$, $N\geq2$,  $F:(0,+\infty) \times (0,+\infty)\times (0,+\infty)^n\times\R_{+}^n \rightarrow\R$ be locally Lipschitz. Let $R:\K^N\rightarrow\R$ be defined by the formula 
\begin{equation}\label{eq:pbapp0} R(K):=F\Big(|K|, \tau(K), \lambda_{1}(K),\ldots,\lambda_{n}(K),\mu_{1}(K),\ldots,\mu_{n}(K)\Big)\end{equation}
where $\tau(K)$ is the torsional rigidity of $K$, $\lambda_{1}(K),\ldots,\lambda_{n}(K)$ are the $n$ first Dirichlet eigenvalues of $K$, and $\mu_{1}(K),\ldots,\mu_{n}(K)$ the $n$ first  Neumann eigenvalues of $K$ (see Section \ref{sect:examples} for  precise definitions).
Let $D\subset \R^N$ be measurable (non-necessarily with finite measure).
\begin{itemize}
    \item Any solution to the problem
 \[\inf\left\{P(K)+R(K),\ K\in\K^N,\ K\subset D\right\}\]
is of class $C^{1,1}$.
\item Suppose that $D$ is a convex body and let $0<V_0<|D|$.  Then any solution to the problem
 \begin{equation}\nonumber \inf\left\{P(K)+R(K),\ K\in \K^N,\ K\subset D, \ |K|=V_0 \right\}\end{equation} 
 is of class $C^{1,1}$.
 \end{itemize} 
\end{theorem}

In other words, we have identified a large class of functions $R$ so that solutions to \eqref{eq:pb1} are smooth up to the $C^{1,1}$-regularity. We refer to Section \ref{ssect:sharp} for an example of a shape optimization problem of the kind \eqref{eq:pb1} leading to an optimal shape that is a stadium, and is therefore not $C^2$, thus showing that our result is sharp in general. Note that existence is not always ensured for the two problems above (taking for instance $R=0$), and we thus prove in Theorem \ref{th:general_existence} the existence of solutions under various general hypotheses on $R$ and $D$.

\paragraph{Motivations:}

Let us give a few motivations for such a result: shape optimization under convexity constraint goes back to the study of Newton's problem of the body of minimal resistance, that can be formulated as
\begin{equation}\label{eq:newton}
\min\left\{J(u)=\int_{D}\frac{dx}{1+|\nabla u|^2},\;\;u:D\to[-M,0]\;\;\textrm{convex}\right\}
\end{equation}
where $D$ is a smooth convex set in $\R^2$, and $M>0$ is given. In this formulation, the graph of $u$ represents the form of a 3-dimensional convex body, and the energy $J$ models the resistance experienced by the body as it moves through a homogeneous fluid with constant velocity in the direction orthogonal to $D$ (in the negative direction in this formulation). The constant $M$ gives a maximal height for the body under study. We refer to \cite{BB,LP01Ane} for more details about this problem, but it is worth noticing that  while one can prove that this problem admits a solution (despite the energy not having a convexity property in $|\nabla u|$), even when $D$ is a disk the solutions are not explicitly known (see \cite{LWZ22Non}) and  even their regularity is not known. It is nevertheless understood that the problem contains a non-convexity structure and that solutions cannot be locally smooth, see for example \cite{LP01Ane}.
Other models with various backgrounds may also involve a convexity constraint, see for example \cite{RC98} with a model in economics.

In the framework of shape optimization, it is interesting to notice that \eqref{eq:pb0}  might have optimizers that are convex (for example the euclidean ball), and in this case the study of \eqref{eq:pb1} is not relevant. Nevertheless, in many situations, \eqref{eq:pb0} may lead to non-convex solutions or even absence of an optimal shape. Let us give two examples of these situations:
\begin{itemize}
\item in \cite{dPV} (see also \cite{DLPV}), the authors study (as in Theorem \ref{th:intro}, $\lambda_{k}$ denotes the $k$th-Dirichlet eigenvalue of the Laplace operator in $\Om$):
$$\min\Big\{P(\Om)+c\lambda_{k}(\Om), \;\;\Om\subset \R^N\Big\}$$
for $k\in\N^*$ and $c>0$, and show that optimal shapes are smooth up to a residual set of co-dimension less than 8 (see \cite[Remark 3.6]{dPV} where it is shown that this problem is equivalent to a constrained formulation). When $N=2$ we have $(P+c \lambda_{k})(\Conv(\Om))\leq 
(P+c \lambda_{k})(\Om)$ so optimal shapes are necessarily convex, but when $N\geq 3$ this argument is not valid anymore. In \cite[Figure 2]{BO} some numerical computations of optimal shapes are done when $N=3$, and one can observe that for some values of $k$ the optimal shapes are not convex,  so that the same problem with a convexity constraint is of interest (see Section \ref{ssect:selected} for more details about this problem).
\item the famous Gamow's liquid drop model leads to the shape optimization problem
\begin{equation}\label{eq:gamow}
\inf\left\{P(\Om)+\int_{\Om}\int_{\Om}\frac{dxdy}{|x-y|},\;\;\Om\subset \R^3, \;|\Om|=V_{0}\right\}
\end{equation}
where $V_{0}\in(0,+\infty)$. It is conjectured that there is a threshold $V^*>0$ such that the ball is a solution if $V_{0}<V^*$, and that there is no solution if $V_{0}>V^*$, see \cite{KM2,J14Iso,LO14Non} for partial results in this direction.  Let us note that the non-existence phenomenon (which is proven for $V_{0}$ large enough, see \cite{LO14Non}) is expected to be due to the splitting of the mass into pieces, and the convexity constraint is  thus violated for such minimizing sequences. It is therefore interesting to wonder about a version of \eqref{eq:gamow} within the class of convex bodies.  In general, as it is easier to get existence within the class of convex bodies (see for example Theorem \ref{th:general_existence}) hence many problems of the type \eqref{eq:pb1} will be of interest if \eqref{eq:pb0} has no solution.
\end{itemize}
Let us also quote two areas of applications to motivate our regularity results:
\begin{itemize}
\item in the study of Blaschke-Santal\'o diagrams for $(P,\lambda_{1},|\cdot|)$ in the class of convex planar sets (see \cite{FL21Bla}), that is to say describing the set
$$\Big\{(x,y)\in\R^2, \;\exists K\in\K^2, \;\;P(K)=x, \;\lambda_{1}(K)=y, \;|K|=1\Big\}$$
 The authors of \cite{FL21Bla} use some regularity theory in shape optimization under convexity constraint in the proof of their main result \cite[Theorem 1.2]{FL21Bla}. Similar results in higher dimension (replacing $\K^2$ with $\K^N$ for $N\geq 3$) are still open problems, and we believe that the tools we develop in this paper can be of help for further investigation in this direction.
\item since the work of Cicalese and Leonardi \cite{CL12Ase}, it is known that regularity theory can help to prove a quantitative version of classical isoperimetric inequalities, see also \cite{BdPV15Fab} where this strategy is the only one (we know of) giving the optimal exponent. It will be worth investigating if one could get new quantitative isoperimetric inequalities in the class of convex sets thanks to our regularity results.
\end{itemize}

\paragraph{State of the art about regularity theory with convexity constraint:}

In the framework of Calculus of variations, one can wonder about the regularity properties of solutions to the following generalization of \eqref{eq:newton}:
\begin{equation}\label{eq:cvconvex}
\min\left\{\int_{\Om}L(x,u(x),\nabla u(x))dx, \;\;u\in X, \;\;u\textrm{ convex}\right\}
\end{equation}
where $\Om$ is a convex set in $\R^N$, $L:\Om\times \R\times \R^N\to\R$ is a Lagrangian and $X$ is a suitable functional space, possibly including boundary constraints. In \cite{C02Cal} the author obtains in particular a $C^1$-regularity result when $L$ is locally uniformly convex in the third variable, and when $N=1$. In \cite{CLR01Reg}, the authors study the same case ($L$ locally uniformly convex), but this time when $N\geq 1$, and $X$ includes a Dirichlet boundary condition: they identify conditions on $\Om, L$ and $X$ so that solutions are $C^1$.

These results were not sharp in general, therefore Caffarelli, Carlier and Lions studied in \cite{CCL}\footnote{At the time we are writing this paper, the work \cite{CCL} is not published.  Let us say here that we will use several ideas from this paper, though we will reproduce them for the convenience of the reader. We try to make it as clear as possible when these ideas are used in our proofs. We warmly thank G. Carlier for providing us a version of \cite{CCL}.} the model case
\begin{equation}\label{eq:model}\min\left\{J(u):=\frac{1}{2}\int_{\Om}|\nabla u|^2dx+\int_{\Om}fudx,\;\;u\in H^1(\Om),\;\;u\textrm{ convex}\right\},
\end{equation}
and proved that a minimizer $u^*$ is locally $C^{1,1-N/p}$ in $\Om$ if $f_{+}\in L^p(\Om)$ with $p>N$, and that this regularity is optimal.

In the framework of shape optimization, D. Bucur proved in  \cite{B03Reg} a $C^1$-regularity result for shape optimization problems with convexity constraint, for functionals involving $\lambda_{k}$ and the volume; this result can be applied for example to
$$\min\Big\{\lambda_{k}(K), \;\;|K|=V_{0}, \;K\in\K^N\Big\}.$$
A sharp regularity result for this problem is still an open problem, though it is expected that optimal shapes are $C^{1,1/2}$ (and that this result is sharp when $k=2$), see \cite{L11Abo}. 

For problems of the kind \eqref{eq:pb1}, \cite{LNP12Reg} shows that under some assumption on $R$ (see Remark \ref{rk:lnp}) and assuming $N=2$, solutions must be $C^{1,1}$.  Comparing it to Theorem \ref{th:intro} above, this result applies to $R(K)=F(|K|,\lambda_{1}(K),\tau(K))$. Therefore, the results we show in the current paper are a generalization of \cite[Theorem 1]{LNP12Reg} to the higher-dimensional case, and to a wider class of functional as well. 

Finally, in \cite{GNR} we can find another $C^{1,1}$-regularity result for solutions to the following 2-dimensional version of a model for charged liquid drop at an equilibrium state:
\begin{multline}\label{eq:gnr}\min\Big\{P(K)+\mathcal{I}(K), \;K\in\K^2, \;|K|=V_{0}\Big\}\\\textrm{ where }\;\;\mathcal{I}(K)=\inf\left\{\int_{K\times K}\log\left(\frac{1}{|x-y|}\right)d\mu(x)d\mu(y), \mu\in\mathcal{P}(K)\right\}
\end{multline} where $\mathcal{P}(K)$ denotes the set of probabily measures supported on $K$.
Here $\mathcal{I}(K)$ can be seen as a capacity term, and it is not so far from the functionals involved in \eqref{eq:pbapp0}, though it is related to a PDE in the exterior of $K$. Our result does not apply directly to \eqref{eq:gnr} or to its higher-dimensional generalizations, but it will be the  subject of future work to adapt our tools to this context.

\paragraph{Strategy of proof and plan of the paper:}

When dealing with regularity theory for \eqref{eq:pb1} or \eqref{eq:cvconvex}, we already have a mild regularity property, namely that solutions are necessarily locally Lipschitz. This is a big difference with \eqref{eq:pb0} where the most difficult part is to prove that solutions are a bit regular, further regularity being obtained usually through an Euler-Lagrange optimality condition.

In \cite{C02Cal} as well as in \cite{LNP12Reg}, the proofs of the regularity results also rely on the writing and the use of an Euler-Lagrange equation, taking into account the convexity constraint, which involves a Lagrange multiplier (infinitely dimensional). It does not seem easy to adapt this method to higher dimensional cases.

In \cite{B03Reg,CLR01Reg,CCL,GNR}, the method is rather different, and consists in building test functions or shapes using a cutting procedure.

In this paper, we obtain three main results, which together lead to the proof of Theorem \ref{th:intro}:
\begin{enumerate}
\item in the spirit of what is done without convexity constraint, we introduce a new notion of quasi-minimizer of the perimeter under convexity constraint (see Definition \ref{def:qmpc}). We show in Theorem \ref{th:shape1} that these sets are $C^{1,1}$ adapting the ideas of \cite{CCL}.  A first important observation is that when writing the perimeter term as a function on the graph, we obtain a Lagrangian of the form $\int_{D}L(\nabla u)dx$ with a uniform convexity property, which explains that the ideas for \eqref{eq:model} can be adapted to this case. 
 However, the main difficulty here is to be careful on how a convex body can be seen as the graph of a convex function: it is not possible to have a local point of view, because this would lead to constraints that are too restrictive  (see Remark \ref{rk:localqpmc}). 
As an application, we show in Corollary \ref{th:shapecor} that if $R$ satisfies a suitable Lipschitz property with respect to the volume metric (see \eqref{eq:hypR3}), then solutions to \eqref{eq:pb1} (a priori with no other constraints) are quasi-minimizer of the perimeter under convexity constraint, and are therefore $C^{1,1}$. These results are described in Section \ref{sect:shape1}.
\item then in Section \ref{sect:shape3}, and in the same spirit with what is done in the {\it classical} case (without convexity constraint,  see \cite{T88Var} and \cite{DLPV} for example), we show how one can handle volume constraint (see Theorem \ref{th:shape3}). To that end we show that the volume constraint can be penalized using Minkowski sums  (see Lemma \ref{lem:shapepenal}).
\item finally in Section \ref{sect:examples} we focus on examples, and in particular we show that the functional \eqref{eq:pbapp0} satisfies a Lipschitz property with respect to the volume metric, so that Theorem \ref{th:shape1} can be applied (see Theorem \ref{th:examples}).
\end{enumerate}

\section{Regularity in shape optimization}\label{sect:shape}
In the classical context of sets minimizing perimeter (without convexity constraint), the concept of quasi-minimizer of the perimeter has proved to be very convenient: denoting $P$ the classical De-Giorgi perimeter, we say that $\Om\subset \R^N$ is a local quasi-minimizer of the perimeter if there exists $\alpha\in(0,1]$, $C>0$ and $r_{0}>0$ such that for every $r\in(0,r_{0})$ and $x\in\R^N$ we have:
\begin{equation}\label{eq:quasimin}
P(\Om)\leq P(\widetilde{\Om})+Cr^{N-1+\alpha}, \;\;\textrm{ for every measurable }\widetilde{\Om}\subset\R^N\textrm{ such that }\Om\Delta\widetilde{\Om}\Subset B_{r}(x).
\end{equation}

The regularity theory then shows that quasi-minimizers of the perimeter are  $C^{1,\alpha/2}$, up to a possibly singular set of dimension less than $N-8$ (see for instance \cite{T88Var}). This regularity can even be strengthened to $C^{1,\alpha}$ for every $\alpha\in(0,1)$ if there exists $\Lambda>0$ such that \begin{equation}\label{eq:strong_quasimin} P(\Om)\leq P(\widetilde{\Om})+\Lambda |\Om\Delta \widetilde{\Om}|,\ \text{ for every measurable }\widetilde{\Om}\subset\R^N\end{equation} (see \cite[Theorem 4.7.4]{Amb}).
In order to take advantage of these results, when studying a shape optimization problem involving the perimeter in the energy functional, one tries to show that a minimizer of our problem must be a quasi-minimizer of the perimeter. To that end, one needs to handle the different terms in the energy, as well as the various constraints.

In this section we therefore introduce a new notion of quasi-minimizer of the perimeter, under a convexity constraint. We study the regularity property it leads to, and then show how one can deal with various constraints and energy terms.
 
Throughout this section we denote by $\K^N$ the class of convex bodies of $\R^N$ (convex compact sets with nonempty interior). Note that (as convex bodies are Lipschitz set), we have $P(K)=\H^{N-1}(\partial K)$ for any $K\in \K^N$,  {\it i.e.} the perimeter of a convex body is the $N-1$-dimensional Hausdorff measure of its boundary.

\subsection{Regularity for quasi-minimizers of the perimeter under convexity constraint}\label{sect:shape1} 

\begin{definition}\label{def:qmpc}We say that $K\in \K^N$ is a quasi-minimizer of the perimeter under convexity constraint if there exist $\eps_K>0, \Lambda_K\geq0$ such that
\begin{equation}\label{eq:qmpc} \forall \widetilde{K}\in \K^N\textrm{ such that } \widetilde{K}\subset K \textrm{ and } |K\setminus\widetilde{K}|\leq\eps_K,\ \ P(K)\leq P(\widetilde{K})+\Lambda_K|K\setminus\widetilde{K}|\end{equation}\end{definition}

\begin{remark}\label{rk:localqpmc}
This notion of quasi-minimizer is not the mere restriction to convex perturbations of the standard notion of quasi-minimizer recalled in \eqref{eq:quasimin} and \eqref{eq:strong_quasimin}:
 \begin{itemize}
 \item first, here we ask that $K$ be minimal in a volume-neighborhood instead of asking it only for sets $\widetilde{K}$ verifying $\widetilde{K}\Delta K\subset B_{r}(x)$ for some $x\in \R^N$ and for small enough $r>0$. This is due to the fact that the latter condition is much too restrictive for convex sets, as it is not always possible to perturbate a convex set $K$ into $\widetilde{K}\in \K^N$ only over some ball $B_r(x)$. For instance, if $K\in \K^2$ is a square with $x\in\partial K$ located inside a segment of $\partial K$, then we can see that if $r$ is small enough any $\widetilde{K}\in \K^2$ such that $\widetilde{K}\Delta K\subset B_{r}(x)$ must be $K$ itself. In fact, the possibility of perturbating a convex set $K$ locally around $x\in\partial K$ is somehow directly connected to some kind of strict convexity of $K$ around $x$. As a consequence, the error term is replaced by the volume of $K\Delta \widetilde{K}$, similarly to what is done in \cite{Mag}.
 \item on the other hand we merely require optimality for the sets $\widetilde{K}$ which perturbate $K$ from the inside, as this will be sufficient to obtain regularity properties (see the proof of Theorem \ref{th:shape1}, where the competitors $K_r\subset K$ are obtained by cutting $K$ by a well-chosen hyperplane). Note that improving the quasi-minimality property by allowing also outward perturbations of $K$ does not lead to better regularity properties in general (as shows the counter-example in Section \ref{ssect:sharp}).
 \end{itemize}
\end{remark} 
 
The regularity result for quasi-minimizers proved in this section is the following. 
\begin{theorem}\label{th:shape1}Let $K$ be a quasi-minimizer of the perimeter under convexity constraint. Then $K$ is $C^{1,1}$.\end{theorem}

As mentioned in the introduction, this leads to a regularity result for minimizer of certain energy having a perimeter term: letting $D\subset \R^N$, we consider the following shape optimization problem \begin{equation}\inf\left\{P(K)+R(K),\ K\in \K^N,\ K\subset D \right\}\label{eq:pbshape}\end{equation}  where $R$ is a shape functional satisfying \begin{equation}\label{eq:hypR3} \ \ \forall K\in \K^N, \exists C_K>0,\;\exists \eps_K>0, \ \forall \widetilde{K} \in \K^N\;\;\textrm{ s.t. }\widetilde{K}\subset K \;\textrm{ and } \ |K\setminus \widetilde{K}|\leq \eps_K,  \ \ R(\widetilde{K})-R(K)\leq C_K|K\setminus \widetilde{K}|\end{equation} Then we have the following easy consequence of Theorem \ref{th:shape1}:
\begin{corollary}\label{th:shapecor} Assume that $R:\K^N\rightarrow\R$ satisfies \eqref{eq:hypR3}. Then any solution $ K^{*}$ of \eqref{eq:pbshape} is $C^{1,1}$. \end{corollary}
\begin{remark}
Note that it may happen that \eqref{eq:pbshape} has no solution even if $D$ is bounded: it is the case for example if $R\equiv 0$. See Theorem \ref{th:general_existence} (i) for an existence result when $D$ is a convex body and there is an additional volume constraint on $K$. 
\end{remark}
\begin{remark}\label{rk:lnp}
In \cite{LNP12Reg} is proved a result similar to Corollary \ref{th:shapecor} in the case $N=2$: more precisely, it is proved (see \cite[Corollary 1]{LNP12Reg}) that if $K^*$ is a solution of \eqref{eq:pbshape} and if $R$ admits a shape derivative at $K^*$ (see \cite[Section 5.9.1]{HP}) which can be represented in $L^p(\partial K^*)$ with $p\in[1,\infty]$, which means that for every $\xi\in W^{1,\infty}(\R^2,\R^2)$,
\begin{equation}\label{eq:hyplnp}
R'(K^*)(\xi)=\lim_{t\to 0}\frac{\big(R((\Id+t\xi)(K^*))-R(K^*)\big)}{t}=\int_{\partial K^*}g\; \xi\cdot\nu_{\partial K^*}d\sigma
\end{equation}
for some function $g\in L^p(\partial K^*)$,
 then $\partial K^*\cap D$ is $W^{2,p}$. In particular when $p=\infty$, this leads to the $C^{1,1}$-regularity as in Corollary \ref{th:shapecor}.
 
We also prove (see \cite[Section 3.2]{LNP12Reg}) that the function $R:K\mapsto F(|K|,\lambda_{1}(K),\tau(K))$ (where $F:\R^3\to\R$ is smooth, and the PDE functionals $\lambda_1$ and $\tau$ are defined at the beginning of Section \ref{sect:pde}) satisfies \eqref{eq:hyplnp} for some $g\in L^\infty(\partial K^*)$, leading to a $C^{1,1}$-optimal shape in that case. Let us conclude with two comments:
\begin{enumerate}
\item 
The proof of \cite[Theorem 1,Corollary 1]{LNP12Reg} is completly different from the one of Theorem \ref{th:shape1} and Corollary \ref{th:shapecor}, as it relies on an Euler-Lagrange equation for \eqref{eq:pbshape} (see \cite[Proposition 1]{LNP12Reg}), and we believe that these ideas are restricted to the 2-dimensional case.
\item As we will see in Section \ref{sect:examples}, assumption \eqref{eq:hypR3} is much more flexible than \eqref{eq:hyplnp} and applies to much more examples, in particular it does not require the existence of a shape derivative.
\end{enumerate}
\end{remark}

\noindent  {\bf Ideas of the proof of Theorem \ref{th:shape1}}: As mentioned in the Introduction, the proof consists in building a framework  enabling to use the ideas of \cite{CCL}, where the authors prove $C^{1,1}$-regularity of the minimizers to the calculus of variation problem \eqref{eq:model}. 
For any $K\in \K^N$ we locally write $\partial K$ near some point $\widehat{x_0}\in \partial K$ as the graph of a convex function $u:\Om\rightarrow\R$, so that the perimeter of $K$ "near this point" is seen as a Lagrangian $\int_{\Om} L(\nabla u)$ where $L$ is locally strongly convex, meaning that $L:\R^N\rightarrow\R$ is smooth and verifies
\[\forall M>0,\ \exists \alpha>0, \ \forall \left(|p|\leq M, |p'|\leq M\right),\ L(p')-L(p)\geq \langle \nabla L(p),p'-p\rangle+\frac{\alpha}{2}|p'-p|^2\]
This gives hope that we can use the procedure from \cite{CCL}, as it is natural to expect that such  an energy behaves like \eqref{eq:model}.  A main difficulty however is to show that the geometrical context actually allows to build competitors in a similar fashion  to \cite{CCL}. If $K$ is a quasi-minimizer in the sense of Definition \ref{def:qmpc}, such competitors will be obtained by setting \begin{equation}\nonumber K_v:=K\cap \text{Epi}(v)\end{equation} for well-chosen convex functions $v:\Om\rightarrow\R$ with $v\geq u$, using the notation Epi$(v)$ for the epigraph of $v$. It is important to notice that it is not possible to work locally ({\it i.e.} picking $\Om$ as a small neighborhood) but we rather have to choose $\Om$ maximal, and this leads to new difficulties  in comparison with \cite{CCL} (mostly linked to the case $N\geq3$, see also Remark \ref{th:rkHolder1}). An important part of Step \textbf{(ii)} of the proof is concerned with addressing this issue.\\

\noindent{\bf Proof of Theorem \ref{th:shape1}:}
Let $K$ be a quasi-minimizer of the perimeter under convexity constraint in the sense of Definition \ref{def:qmpc}. 

\noindent{\bf Representation of $K$ as a graph: }Let $\widehat{x_0}\in \partial K$ ;  applying Proposition \ref{th:appgraph}, we get that there exists a hyperplane $H\subset \R^N$ and a unit vector $\xi\in\R^N$ normal to $H$ such that, denoting by $(x,t)$ a point in $H\times\R\xi$ coordinates (and hence denoting $\widehat{x_0}:=(x_0,0)$): \begin{itemize} \item The set $\Om:=\left\{x\in H, \ (x+\R\xi)\cap \text{Int}(K) \ne \emptyset\right \}$ is open, bounded and convex, and the function \begin{align}\nonumber u:\Om&\to\R \\ x&\mapsto \min\{t\in\R,\ (x,t)\in K\}\nonumber\end{align} is well-defined and convex. \item It holds\begin{align}\nonumber\{(x,u(x)), \ x\in \Om\}&\subset \partial K \\ K\cap (\Om\times\R\xi)&\subset\{(x,t)\in\Om\times\R\xi,\ u(x)\leq t\}\nonumber\end{align}\item There exists a $\beta>0$ and $c:=c(\beta)>0$ such that $B_{\beta}:=B_{\beta}(x_0)\Subset \Om$ verifies \begin{equation}\label{eq:cbeta} \{(x,t)\in B_{\beta}\times\R\xi,\ u(x)\leq t\leq u(x)+c\}\subset K.\end{equation}\end{itemize}
Throughout the proof the coordinates $(x,t)$ are thought in the orthogonal decomposition $H\times \R\xi$. Moreover the notation $B_r(x)$ for some $x\in H$ and $r>0$ will denote a ball lying in $H$.

Since $B_{\beta}\Subset \Om$ we have that $u$ is globally Lipschitz in $B_{\beta}$. Setting some $y\in B_{\beta/2}(x_0)$ and $p\in \partial u(y)$ (where $\partial u(y)$ denotes the subdifferential of the convex function $u$ at $y$) we let $l(x):=u(y)+\langle p,x-y\rangle$  for $x\in H$ and $M_r:=\sup_{B_r(y)}(u-l)$. We aim to prove that there exists $C>0$  and $r_0>0$ (both independent on $y$) such that $M_r\leq Cr^2$ for any $y\in B_{\beta/2}(x_0)$ and  $0<r<r_0$. This classically ensures that $u$ is $C^{1,1}$ over $B_{\beta/2}(x_0)$ (see for instance Lemma 3.2 in \cite{DF}). As the case $M_{r}=0$ is trivial, from now on we fix $y\in B_{\beta/2}(x_0)$, $0<r<\beta/2$ and assume $M_{r}>0$. Note that $p$, $l$, $M_r$ (and other objects we will introduce along the proof) depend on $y$, although for simplicity it does not appear in the notations. We will also set $y=0$ for simplicity, while paying attention to the fact that the estimates we make in the proof do not depend on $y$. \\

\noindent{\bf Construction of a competitor: } Let $q_r$ be some unit vector such that $M_r=(u-l)(rq_r)$. We set 
 \begin{equation}\forall x\in H, \ \sigma_r(x):=l(x)+\frac{M_r}{2r}(\langle q_r,x\rangle+r),\ \widehat{\sigma}_r(x):=(x,\sigma_r(x)),\nonumber \end{equation} 
\begin{equation} H_r:=\widehat{\sigma}_r(H), \ H_r^+:=\{(x,t)\in \R^{N-1}\times\R, \ t\geq \sigma_r(x)\},\nonumber\end{equation} and finally we define:
\begin{equation} K_r:=K\cap H_r^+\nonumber\end{equation} 
Notice that $K_r\subset K$ is convex and compact. As we will show  in section {\bf (i)} of the proof, the construction of $K_r$ ensures that Int$(K_r)\ne \emptyset$ and $|K\setminus K_r|\leq\eps_K$ for $r\leq r(\text{diam}(\Om),\eps_K,\|\nabla u\|_{L^{\infty}(B_{\beta})})$
(see \eqref{eq:finalvol}  and the end of section {\bf (i)}). Therefore from \eqref{eq:qmpc} we  will get \begin{equation}\label{eq:shape_variation} P(K)-P(K_r)\leq \Lambda_K|K\setminus K_r|\end{equation} for such $r$. With \eqref{eq:shape_variation} we are left to estimate {\bf (i)} the volume variation from above and {\bf (ii)} the perimeter variation from below.  We first provide some central estimates on the size of the set $\{u\leq \sigma_r\}$ which were proven in \cite{CCL}. Note that these estimates only use convexity of $u$ and do not depend on the particular kind of energy introduced in the problem \eqref{eq:model}. We reproduce the proof of \cite{CCL} below for the convenience of the reader.\\

 \noindent{\bf Estimate of $\{u\leq\sigma_r\}$:} Let us prove\begin{equation}\label{eq:omegar} B_{r/2}^+(0)\subset \{u\leq\sigma_r\}\subset \{|\langle q_r,\cdot\rangle|\leq r\}\end{equation} where we set $B_{r/2}^+(0)=B_{r/2}(0)\cap\{\langle q_r,\cdot\rangle\geq 0\}$. \begin{itemize}
\item {$\{u\leq\sigma_r\}\subset \left\{\langle q_r,\cdot\rangle\geq -r\right\}$:} if $\langle q_r,x\rangle<-r$, then \begin{equation}\nonumber u(x)\geq l(x)> l(x)+\frac{M_r}{2r}(\langle q_r,x\rangle+r)=\sigma_r(x)\end{equation} whence we deduce $x\notin\{u\leq\sigma_r\}$. \item {$\{u\leq\sigma_r\}\subset \left\{\langle q_r,\cdot\rangle\leq r\right\}$:} over the interval $I:=\Om\ \cap\left\{tq_r, t>r\right\}$ we know that $u>l+M_r$ thanks to the convexity of $u$. Therefore one can separate the convex sets $I$  and $\{u\leq l+M_r\}$ by some hyperplane $\Pi$. Since by definition $B_r(0)\subset \{u\leq l+M_r\}$, $\Pi$ must also seperate $B_r(0)$ and $I$, implying that $\Pi=\{x\in\R^N,\ \langle q_r,x\rangle=r\}$. This yields in particular $(u-l)\geq M_r$ over $\Pi \cap \Om$. Given now $x \in \Om$ such that $\langle q_r,x\rangle>r$ set $z \in\Pi\cap[0,x]$: from the two informations $(u-\sigma_r)(0)=-M_r/2$ and $(u-\sigma_r)(z)=(u-l)(z)-M_r\geq 0$ we deduce $(u-\sigma_r)(x)>0$ using convexity of $u-\sigma_r$, so that $x\notin \{u\leq\sigma_r\}$. \item $B_{r/2}^+(0)\subset \{u\leq\sigma_r\}$: given $x\in B_{r/2}(0)$, we have $2x\in B_r(0)$ hence \begin{equation}\nonumber(u-l)(x)\leq \frac{1}{2}(u-l)(2x)+\frac{1}{2}\underbrace{(u-l)(0)}_{=0}\leq \frac{M_r}{2}\end{equation} using convexity. If in addition $\langle q_r,x\rangle\geq 0$ then \begin{equation}\nonumber(u-\sigma_r)(x)=(u-l)(x)-\frac{M_r}{2r}(\langle q_r,x\rangle+r)\leq \frac{M_r}{2}-\frac{M_r}{2}\leq 0\end{equation}so that $x\in \{u\leq\sigma_r\}$. 
\end{itemize}

\begin{center} \textbf{(i) Estimate from above} \end{center} 
We have 
\begin{align}|K\setminus K_r|&=|K\cap(\R^N\setminus H_r^+)| \nonumber\\\nonumber&=\left|\{(x,t)\in K,\ t< \sigma_r(x)\}\right|\\\nonumber &\leq|\{(x,t)\in \Om\times \R^+,\ u(x)\leq t\leq \sigma_r(x)\}|\\&= \int_{\{u\leq \sigma_r\}}(\sigma_r-u)d\H^{N-1}\label{Fub_vol}\end{align}
 using Fubini's Theorem. 

If $x\in\{u\leq\sigma_r\}$ we have thanks to the right-hand-side inclusion of \eqref{eq:omegar}\begin{align}\nonumber0\leq (\sigma_r-u)(x) = \underbrace{(l-u)(x)}_{\leq0}+\frac{M_r}{2r}\left(\langle q_r,x\rangle+r\right)\leq M_r \nonumber\end{align}\textit {i.e.} \begin{equation}\label{eq:u_sig_above2} 0\leq \sigma_r -u\leq M_r \text{ over } \{u\leq\sigma_r\}\end{equation}  \\ Injecting \eqref{eq:u_sig_above2} into \eqref{Fub_vol} yields
\begin{equation} \label{eq:finalvol} |K\setminus K_r|\leq M_r\H^{N-1}(\{u\leq\sigma_r\})\end{equation} 
 We will refine further \eqref{eq:finalvol} (into \eqref{eq:finalvol2}), but this estimate is sufficient for now. It gives in particular that there exists $r_{0}(\text{diam}(\Om),\eps_K,\|\nabla u\|_{L^{\infty}(B_{\beta})})>0$ such that $|K\setminus K_r|\leq \eps_K$ for any $r<r_{0}(\text{diam}(\Om),\eps_K,\|\nabla u\|_{L^{\infty}(B_{\beta})})$ hence also that Int$(K_r)\ne \emptyset$ for such $r$. Indeed, as $|p|\leq \|\nabla u\|_{L^{\infty}(B_{\beta})}$ it holds for $0<r<\beta/2$:
\begin{equation} M_r\leq \sup_{x\in B_r(0)}|u(x)-u(0)|+\sup_{x\in B_r(0)}|\langle p,x\rangle|\leq 2r\|\nabla u\|_{L^{\infty}(B_{\beta})} \label{eq:Mr/2r}\end{equation} so that from \eqref{eq:finalvol} we get \[|K\setminus K_r|\leq 2r\|\nabla u\|_{L^{\infty}(B_{\beta})}\H^{N-1}(\{u\leq\sigma_r\})\leq 2r\|\nabla u\|_{L^{\infty}(B_{\beta})}\H^{N-1}(\Om)\]which yields $|K\setminus K_r|\leq\eps_K$ for $r<\eps_K/(2\H^{N-1}(\Om)\|\nabla u\|_{L^{\infty}(B_{\beta})})$.

\begin{center} \textbf{(ii) Estimate from below} \end{center} 
We now deal with estimating from below the perimeter variation.  In the view of \eqref{eq:finalvol}, if $r\leq  r_{0}(\text{diam}(\Om),\eps_K,\|\nabla u\|_{L^{\infty}(B_{\beta})})$ then $K_r$ has non-empty interior, which we will suppose from now on. Let us start by showing that \begin{equation}P(K)-P(K_r)\geq \H^{N-1}(\widehat{u}(\Om))-\H^{N-1}(\widehat{v_r}(\Om)\cap \partial K_r) \label{eq:shapeperim}\end{equation} where we set \begin{equation}\nonumber\forall x\in \Om, \ \widehat{u}(x):=(x,u(x)), \ v_{r}(x)=\max\{u,\sigma_{r}\}(x), \; \widehat{v_r}(x):=(x,v_{r}(x)).\end{equation}  
We have
 \begin{eqnarray*}
P(K)-P(K_r)&=&\H^{N-1}(\partial K)-\H^{N-1}(\partial K_r)\\
&=&\H^{N-1}(\partial K\cap \widehat{u}(\Om))+\H^{N-1}(\partial K\cap \widehat{u}(\Om)^c)-\H^{N-1}(\partial K_{r}\cap \widehat{v_{r}}(\Om))\\&&-\H^{N-1}(\partial K_{r}\cap \widehat{v_{r}}(\Om)^c) \\ &=&\H^{N-1}(\widehat{u}(\Om))-\H^{N-1}(\partial K_{r}\cap \widehat{v_{r}}(\Om))+\H^{N-1}(\partial K\cap \widehat{u}(\Om)^c)\\&&-\H^{N-1}(\partial K_{r}\cap \widehat{v_{r}}(\Om)^c)
\end{eqnarray*} using in the third line that $\widehat{u}(\Om)\subset \partial K$.
If we show 
\begin{equation}\label{eq:inclbelow} \partial K_r\cap \widehat{v_r}(\Om)^c\subset \partial K\cap\widehat{u}(\Om)^c\end{equation}
then we obtain \eqref{eq:shapeperim}. 
Therefore, let $\widehat{x}\in \partial K_r\cap \widehat{v_r}(\Om)^c$. As $\partial K_{r}\subset K$, we first want to show that $\widehat{x}\notin \Int(K)$. Assume by contradiction that $\widehat{x}\in \Int(K)$, then as
\begin{equation}\label{eq:Krdecomp} \partial K_r=\partial (K\cap H_{r}^+)=(K\cap \partial H_{r}^+)\cup (\partial K\cap H_{r}^+)=(K\cap H_r)\cup\left(\partial K\cap H_r^+\setminus H_r\right)\end{equation}
(the second equality comes from the fact that $K_{r}$ and $H_{r}^{+}$ are closed)
 we must have $\widehat{x}\in H_r$. But then $\widehat{x}\in H_r\cap\text{Int}(K)$ and we deduce that there exists $x\in \Om$ such that $\widehat{x}=(x,\sigma_r(x))$ with $\sigma_r(x)>u(x)$, thus getting $v_r(x)=\sigma_r(x)$ and $\widehat{x}=\widehat{v_r}(x)$, which is a contradiction. 
Now, as $\widehat{x}\in \partial K_r\subset H_r^+$, assuming by contradiction that there exists $x\in \Om$ such that $\widehat{x}=(x,u(x))$ leads to $u(x)\geq\sigma_r(x)$, implying again the contradiction $\widehat{x}=\widehat{v_r}(x)$. This  concludes the proof of \eqref{eq:inclbelow} and \eqref{eq:shapeperim}.

Let us rewrite \begin{align}\label{eq:vr(Omr)}\widehat{v_r}(\Om)\cap\partial K_r&=\widehat{v_r}(\Om)\cap K=\widehat{v_r}(\Om_r)\end{align} by setting $\Om_r:=\widehat{v_r}^{-1}(K)\subset \Om$ (see Figure \ref{fig:cutting}).  The first equality of \eqref{eq:vr(Omr)} is justified in the following way: first, as $\partial K_r\subset K$, then $\widehat{v_r}(\Om)\cap\partial K_r\subset \widehat{v_r}(\Om)\cap K$. Second, if $\widehat{x}\in \widehat{v_r}(\Om)\cap K$, let us write $\widehat{x}=(x,v_r(x))$ for some $x\in \Om$. Then either $\sigma_r(x)\geq u(x)$, giving $\widehat{x}=\widehat{v_r}(x)=\widehat{\sigma_r}(x)\in H_r$ ; as $\widehat{x}\in K$, we get that $\widehat{x}\in K\cap H_r\subset \partial K_r$ thanks to \eqref{eq:Krdecomp}. Else, $u(x)> \sigma_r(x)$ so that $\widehat{x}=\widehat{u}(x)\in \partial K\cap H_r^+\setminus H\subset \partial K_r$ using again \eqref{eq:Krdecomp}. \\ We then get from \eqref{eq:shapeperim} \begin{equation} P(K)-P(K_r)\geq \H^{N-1}(\widehat{u}(\Om))-\H^{N-1}(\widehat{v_r}(\Om_r)) \geq\H^{N-1}(\widehat{u}(\Om_r))-\H^{N-1}(\widehat{v_r}(\Om_r))\label{eq:shapeperim2}\end{equation}
We will rewrite the right hand side of \eqref{eq:shapeperim2} using the classical formula for the perimeter of a Lipschitz graph, but we start  
by showing two importants features of $\Om_r$. Let us note here that the introduction of $\Om_r$ is not necessary if $N=2$, while it is meaningful for $N\geq3$ (see Remark \ref{th:rkHolder1}).
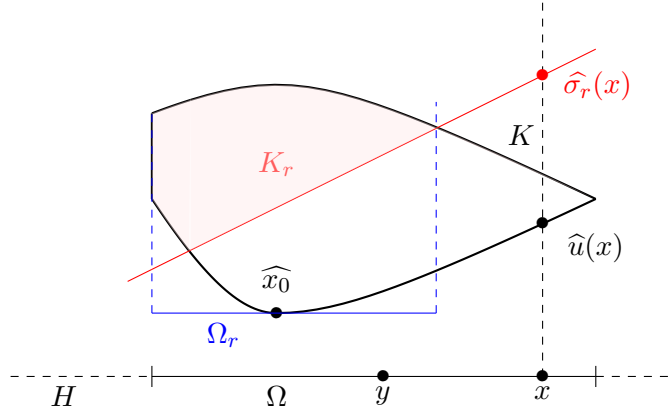
\begin{figure}
\centering
 \begin{tikzpicture}[scale=1.4]
\draw(-1.172,-0.6) -- (3, -0.6);
\draw (-1.172,-0.5) -- (-1.172,-0.7);
\draw(3,-0.5) -- (3,-0.7);
\draw(0, -0.6)node[below]{$\Omega$};
\draw[thick,domain=-1.172:-0.82, variable=\x, name path = h] plot ({\x}, {2*(sqrt(1+\x*\x)-1)});
\draw[thick,domain=-0.82:0, variable=\x] plot ({\x}, {2*(sqrt(1+\x*\x)-1)});
\draw[thick,domain=0:3, variable=\x] plot ({\x}, {0.5*(sqrt(1+\x*\x)-1)});
\draw[thick,domain=-0.81:3, name path = g, variable=\x] plot ({\x}, {2*1.081-0.5*(sqrt(1+\x*\x)-1)});
\draw[thick,domain=-1.172:-0.81, name path = i, variable=\x] plot ({\x}, {2*1.081-0.5*(sqrt(1+\x*\x)-1)});
\draw[thick]  (-1.172,1.081) -- (-1.172,1.8917);

\draw (0,0) node{$\bullet$} ;
\draw (0,0.1) node[above]{$\widehat{x_0}$};
\draw (2.3,1.5) node[above]{$K$} ;
\draw[red] (0,1.2) node[above]{$K_r$} ;
\draw (2.5,-0.6) node[below]{$x$};
\draw (2.5,-0.6) node{$\bullet$} ;
\draw (1,-0.6) node[below]{$y$}  ;
\draw (1,-0.6) node{$\bullet$} ;
\draw (-2,-0.6) node[below]{$H$};
\draw[dashed] (2.5,-0.6) -- (2.5,3);
\draw[dashed] (-2.5,-0.6) -- (-1.172,-0.6);
\draw[dashed] (3,-0.6) -- (3.7,-0.6);
\draw[blue] (-1.172,0) -- (1.5,0);
\draw[red](2.6,2.15) node[right]{$\widehat{\sigma_r}(x)$};
\draw[red](2.5,2.25) node{$\bullet$};
\draw(2.5,0.85) node{$\bullet$};
\draw(3,0.85) node[below]{$\widehat{u}(x)$};
\draw[red,domain=-1.4:-1.172]plot({\x},{0.5*\x+1});
\draw[red,domain=-1.172:-0.81]plot({\x},{0.5*\x+1});
\draw[red,domain=1.5:3]plot({\x},{0.5*\x+1});
\draw[red,domain=-0.81:1.5, name path = f]plot({\x},{0.5*\x+1});
\tikzfillbetween[of=f and g]{red!10, opacity=0.4}; 
\tikzfillbetween[of=h and i]{red!10, opacity=0.4}; 
\draw[blue] (-0.5,0) node[below] {$\Omega_r$};
\draw[blue,dashed] (-1.172,0) -- (-1.172,1.9);
\draw[blue,dashed] (1.5,0) -- (1.5,2);
\end{tikzpicture}
\caption{\label{fig:cutting} Cutting procedure}
\end{figure}

\noindent{\bf $\Om_r$ is convex:} 
Let us show that $\Om_r=\pi_H(K_r)$, where $\pi_H(K_r)$ is the orthogonal projection over $H$ of the convex set $K_r$, which will give right away that $\Om_r$ is convex. First, if $x\in \Om_r$ then $x\in \Om$ with $(x,v_r(x))\in K$, and $t=v_r(x)\geq\sigma_r(x)$ satisfies that $(x,t)\in K$ with $t\geq\sigma_r(x)$, providing $(x,t)\in K_r$ hence $x\in \pi_H(K_r)$. Conversely, let $x\in H$ be such that there exists $t\in\R$ with $(x,t)\in K_r$, implying that $(x,t)\in K$ with $t\geq \sigma_r(x)$. Note that $x\in \Om$ by definition of $\Om$ and using that $K_r\subset K$. If $u(x)\geq\sigma_r(x)$ then $\widehat{v_r}(x)=\widehat{u}(x)\in K$ which means that $x\in \Om_r$. Else, $t$ is such that $t\geq\sigma_r(x)\geq u(x)$, which gives that $\widehat{v_r}(x)=\widehat{\sigma_r}(x)\in K$ by convexity of $K$, since $(x,t)\in K$ and $(x,u(x))\in K$. Thus $\Om_r=\pi_H(K_r)$ and $\Om_r$ is convex.

%

\noindent{\bf $\Om_r$ has non-empty interior:} 
We now prove that $\Om_r$ has non-empty interior with a ball which has size uniform in $y$ (which we set to be $y=0$), {\it i.e.} that there exists $\widetilde{\beta}\in(0,\frac{\beta}{2})$ such that
\begin{equation}\label{eq:BbetaOmr}\forall r\in (0,\widetilde{\beta}), \ B_{\widetilde{\beta}}(0)\subset \Om_r\end{equation} 
Given $\widetilde{\beta}\in(0,\beta/2)$ that will be chosen later, using \eqref{eq:Mr/2r} and $|p|\leq \|\nabla u\|_{L^{\infty}(B_{\beta})}$ we get for any $x\in B_{\widetilde{\beta}}(0)$ and $r\in(0,\widetilde{\beta})$, 
  \begin{align}\nonumber 
  \sigma_r(x)&=u(0)+\langle p,x\rangle+\frac{M_r}{2r}(\langle q_r,x\rangle+r)\\
  \nonumber &\leq u(0)+\|\nabla u\|_{L^{\infty}(B_{\beta})}\left(2\widetilde{\beta}+ r\right)
\leq u(x)+4\widetilde{\beta}\|\nabla u\|_{L^{\infty}(B_{\beta})}\end{align}  
If we choose $\widetilde{\beta}$ such that $4\widetilde{\beta}\|\nabla u\|_{L^{\infty}(B_{\beta})}\leq c(\beta)$  where $c(\beta)$ satifies \eqref{eq:cbeta}, we deduce 
  \begin{equation}\label{eq:sigmarbeta}\forall r\in (0,\widetilde{\beta}),\ \sigma_r\leq u+c \text{ in } B_{\widetilde{\beta}}(0)\end{equation} 
  We are now in a position to prove \eqref{eq:BbetaOmr}. 
  Let $x\in B_{\widetilde{\beta}}(0)$: if $u(x)\geq\sigma_r(x)$, then $\widehat{v_r}(x)=\widehat{u}(x)\in K$ hence $x\in \Om_r$. 
  Otherwise $u(x)\leq\sigma_r(x)$, and then $u(x)\leq \sigma_r(x)\leq u(x)+c$ thanks to \eqref{eq:sigmarbeta}; using \eqref{eq:cbeta} we get in fact $\widehat{v_r}(x)=(x,\sigma_r(x))\in K$, meaning that $x\in \Om_r$. \\

\noindent {\bf Rewriting \eqref{eq:shapeperim2} with Lipschitz graphs}: We claim now that \eqref{eq:shapeperim2} rewrites 
   \begin{align}\label{eq:changeLip} P(K)-P(K_r)&\geq \int_{\Om_r}\left[\sqrt{1+|\nabla u|^2}-\sqrt{1+|\nabla v_r|^2}\right]d\H^{N-1}\\& = \int_{\text{Int}(\Om_r)}\left[\sqrt{1+|\nabla u|^2}-\sqrt{1+|\nabla v_r|^2}\right]d\H^{N-1}\nonumber \end{align}
   If \eqref{eq:changeLip} holds true, the second line comes from the fact that $|\Om_r|=|\text{Int}(\Om_r)|$, since $\Om_r$ is convex. Now, from \cite[Remark 2.72]{AFP} one has 
   \begin{equation}\H^{N-1}(\widehat{u}(\omega))=\int_{\omega}\sqrt{1+|\nabla u|^2}d\H^{N-1}\nonumber\end{equation}
   if $\omega\subset \Om_r$ with $u_{|\omega}$ Lipschitz. As $u$ is not necessarily Lipschitz over the whole of $\Om_r$, let us take an increasing sequence $(\Om_n)$ of open subsets of $\Om$ with $\Om_n\Subset \Om$ for each $n$ and $\cup_n\Om_n=\Om$. 
   Then setting $\Om_r^n:=\Om_n\cap \Om_r$, the sequence $(\Om_r^n)$ is still increasing with $\cup_n\Om_r^n=\Om_r$. 
   As $u_{|\Om_r^n}$ is now Lipschitz we can write
   \begin{equation}\H^{N-1}(\widehat{u}(\Om_r^n))=\int_{\Om_r^n}\sqrt{1+|\nabla u|^2}d\H^{N-1}\nonumber\end{equation}
   The monotonous convergence theorem applies on each side of the equation, yielding at the limit
   \begin{equation}\H^{N-1}(\widehat{u}(\Om_r))=\int_{\Om_r}\sqrt{1+|\nabla u|^2}d\H^{N-1}\label{eq:changeLipu}\end{equation}
   The same goes for $\H^{N-1}(\widehat{v_r}(\Om_r))$, thus getting \eqref{eq:changeLip}.\\
   
\noindent {\bf Estimate from below of \eqref{eq:changeLip}:} Let 
   \begin{equation}\nonumber \omega_r:=\{x\in \text{Int}(\Om_r),\ u(x)\leq \sigma_r(x)\} \subset\{u\leq\sigma_r\}\end{equation} 
   and $\widetilde{\omega_r}$ be the projection of $\omega_r$ onto $\Gamma_r:=\{\langle q_r,\cdot\rangle =0\}$.

 Thanks to \eqref{eq:omegar} we have\begin{equation}\label{eq:upomegar} \omega_r\subset \{|\langle q_r,\cdot\rangle|\leq r\}\end{equation}
On the other hand, if $0<r<\widetilde{\beta}$ we have $B_r(0)\subset \Om_r$ using \eqref{eq:BbetaOmr}, so that \eqref{eq:omegar} again gives \begin{equation}\label{eq:downomegar} B^+_{r/2}(0)\subset \omega_r\end{equation} for such $r$.
     
The local strong convexity of the function $\xi\in \R^N\mapsto \sqrt{1+\xi^2}$ combined with the fact that for any $r\in(0,\beta/2)$\begin{equation}\nonumber|\nabla \sigma_r| =\left|p+\frac{M_{r}}{2r}q_{r}\right| \leq 2\|\nabla u\|_{L^{\infty}(B_{\beta})}\end{equation} 
(where we used \eqref{eq:Mr/2r})  enable to find $\alpha=\alpha(\|\nabla u\|_{L^{\infty}(B_{\beta})})>0$ such that for any $r\in (0,\beta/2)$, 
      \begin{equation}\nonumber\sqrt{1+|\nabla u|^2}-\sqrt{1+|\nabla v_r|^2}\geq \frac{\nabla v_r\cdot (\nabla u-\nabla v_r)}{\sqrt{1+|\nabla v_r|^2}}+\frac{\alpha}{2}|\nabla u-\nabla v_r|^2\ \text{over } B_{\beta}\end{equation}
       Note also the weaker (but global) estimate \begin{equation}\sqrt{1+|\nabla u|^2}-\sqrt{1+|\nabla v_r|^2}\geq \frac{\nabla v_r\cdot (\nabla u-\nabla v_r)}{\sqrt{1+|\nabla v_r|^2}}\ \text{ over } \Om_r\nonumber\end{equation}Since $\H^{N-1}(\widehat{u}(\Om_r))\leq \H^{N-1}(\partial K)<+\infty$, \eqref{eq:changeLipu} implies in particular that $\sqrt{1+|\nabla u|^2}\in L^1(\Om_r)$, hence also that $|\nabla u|\in L^1(\Om_r)$. This ensures $u\in W^{1,1}(\Om_r)$ as we also have $u\in L^{\infty}(\Om_r)$. Let us integrate the two previous estimates and use \eqref{eq:changeLip} to get 
      \begin{equation}\label{eq:belowalmostfinal} P(K)-P(K_r)\geq \int_{\omega_r}\frac{\nabla \sigma_r\cdot (\nabla u-\nabla \sigma_r)}{\sqrt{1+|\nabla \sigma_r|^2}}d\H^{N-1}+\frac{\alpha}{2}\int_{\omega_r\cap B_{\beta}}|\nabla u-\nabla \sigma_r|^2d\H^{N-1}\end{equation} 
      Recalling \eqref{eq:u_sig_above2} we have, 
      \begin{equation}\label{eq:u_sig_above} 0\leq\sigma_r-u\leq M_r\text{ over } \omega_r.\end{equation} 
      As $\nabla \sigma_{r}$ is a fixed vector we get by integrating by parts
      \begin{equation}\nonumber -\int_{\omega_r}\frac{\nabla \sigma_r\cdot (\nabla u-\nabla \sigma_r)}{\sqrt{1+|\nabla \sigma_r|^2}} d\H^{N-1} = -\int_{\partial \omega_r}\frac{\nabla \sigma_r\cdot n}{\sqrt{1+|\nabla \sigma_r|^2}}(u- \sigma_r)d\H^{N-2}\leq \int_{\partial \omega_r}(\sigma_r-u)d\H^{N-2}\end{equation} 
      where $n$ denotes the outer unit normal to $\partial \omega_r$. Observing that \eqref{eq:u_sig_above} must also hold on $\partial \omega_r$, and since $u=\sigma_{r}$ on $\partial\omega_{r}\cap \Int(\Om_{r})$ we deduce
      \begin{equation}-\int_{\omega_r}\frac{\nabla \sigma_r\cdot (\nabla u-\nabla \sigma_r)}{\sqrt{1+|\nabla \sigma_r|^2}}d\H^{N-1}\leq M_r\H^{N-2}\big[\partial \omega_r\cap \partial \Om_r\big]\nonumber
      \end{equation}
      Moreover, let us show that,
     \begin{equation} \H^{N-2}\big[\partial \omega_r\cap \partial \Om_r\big]\leq \frac{(N-1)\H^{N-1}(\omega_r)}{d(y,\partial \Om_r)}\label{eq:Stokes}\end{equation} so that for $r\in (0,\widetilde{\beta})$ we get \begin{equation}-\int_{\omega_r}\frac{\nabla \sigma_r\cdot (\nabla u-\nabla \sigma_r)}{\sqrt{1+|\nabla \sigma_r|^2}}d\H^{N-1} \leq M_r\frac{(N-1)\H^{N-1}(\omega_r)}{d(y,\partial \Om_r)}\leq M_r\frac{(N-1)}{\widetilde{\beta}}\H^{N-1}(\omega_r)\leq CM_rr\H^{N-2}(\widetilde{\omega_r})\label{eq:lineartermvol}\end{equation}
  with $C= \frac{2(N-1)}{\widetilde{\beta}}>0$
  , where we used 
  \eqref{eq:BbetaOmr}
   and \eqref{eq:upomegar}. 

 Estimate \eqref{eq:Stokes} was proven in \cite{CCL}; we reproduce the argument for the convenience of the reader: from Stokes formula,  \begin{align}\nonumber\H^{N-1}(\omega_{r})=\frac{1}{N-1}\int_{\omega_{r}}\div(x)d\H^{N-1}(x)&=\frac{1}{N-1}\int_{\partial\omega_{r}}\underbrace{\langle x,n(x)\rangle}_{\geq 0}d\mathcal{H}^{N-2}(x)\\&\geq \frac{1}{N-1}\int_{\partial\omega_{r}\cap\partial\Om_r}\langle x,n(x)\rangle d\mathcal{H}^{N-2}(x)\nonumber\end{align} and notice that for any $x\in \partial \omega_r\cap\partial \Om_r$ such that $n(x)$ is well-defined we have $\langle x,n(x)\rangle=d(0,H_x)$ where $H_x$ is the tangent hyperplane at $x$ to $\partial \omega_r$, which is also the tangent hyperplane at $x$ to $\partial \Om_r$. This gives $\langle x,n(x)\rangle\geq d(0,\partial \Om_r)$ using the convexity of $\Om_r$, implying \eqref{eq:Stokes}.

\begin{figure}
\centering
\caption{\label{fig:omegar} Localization of $\omega_r$}\includegraphics{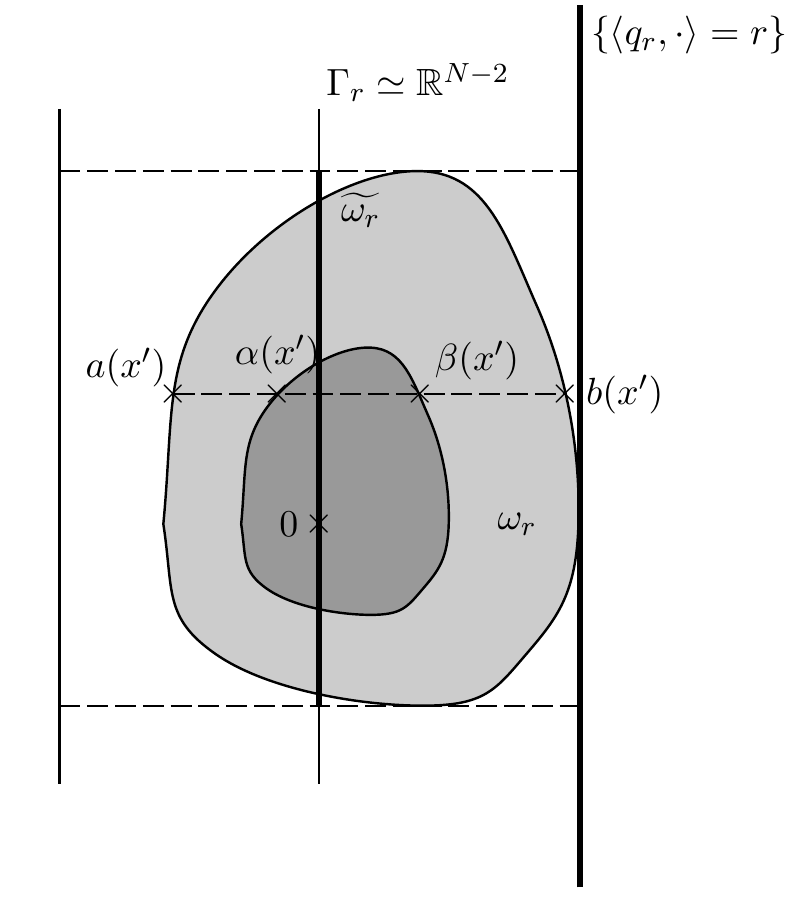}
\end{figure}

\noindent{\bf $L^2$ Estimate from below of $\nabla (u-\sigma_r)$:} We now estimate from below the term 
\begin{equation}\int_{\omega_r\cap B_{\beta}}|\nabla u-\nabla \sigma_r|^2 d\H^{N-1}\geq \int_{\omega_r\cap B_{\widetilde{\beta}}(0)}|\nabla u-\nabla \sigma_r|^2 d\H^{N-1}\nonumber\end{equation} We shall prove\begin{equation}\label{eq:shapebelowperim}\forall r\in (0,\widetilde{\beta}/2),\ \int_{\omega_r\cap B_{\widetilde{\beta}}(0)}|\nabla u-\nabla \sigma_r|^2\geq \delta\frac{M_r^2}{r}\H^{N-2}(\widetilde{\omega_r})\end{equation} for some $\delta=\delta(\widetilde{\beta},\text{diam}(\Om))$. This estimate was proved in \cite{CCL} but we reproduce the proof for the convenience of the reader. \\ \\ Let $\gamma=\frac{\widetilde{\beta}}{2\diam(\Om)}>0$ be such that if $r\in (0,\widetilde{\beta}/2)$ we have thanks to \eqref{eq:BbetaOmr} \begin{equation}\label{eq:incl_gamma} [-r,r]q_r\times \gamma\widetilde{\omega_r}\Subset B_{\widetilde{\beta}}(0)\subset \text{Int}(\Om_r).\end{equation}where $[-r,r]q_r\times \gamma\widetilde{\omega_r}:=\{tq_r+\gamma x',\ t\in[-r,r], \ x'\in \widetilde{\omega_r}\}$. Setting \begin{equation}\nonumber A_{\gamma/2}:=\frac{\gamma}{2}\widetilde{\omega_r},\end{equation} then for any $x'\in A_{\gamma/2}$ we can write \[\begin{cases}(x'+\R q_r)\cap \omega_r=[(a(x'),x'),(b(x'),x')]\\ (x'+\R q_r)\cap \frac{1}{2}\omega_r=[(\alpha(x'),x'),(\beta(x'),x')]\end{cases}\]for some functions $a,b$ and $\alpha,\beta$ defined over $A_{\gamma/2}$, with $(a(x'),x')$ and $(b(x'),x') \in \partial \omega_r\cap B_{\widetilde{\beta}}(0)$ thanks to \eqref{eq:incl_gamma} and the right inclusion of \eqref{eq:omegar} (see Figure \ref{fig:omegar}).  In particular it holds $(u-\sigma_r)(a(\cdot),\cdot)=(u-\sigma_r)(b(\cdot),\cdot)=0$.

Since $u\leq \sigma_r$ over $\omega_r$ and since $u-\sigma_r$ is convex we get for every $h\in \omega_{r}$
 \begin{equation}\nonumber
 (u-\sigma_r)(\frac{h}{2})\leq \frac12(u-\sigma_r)(0)+\frac12(u-\sigma_r)(h)\leq \frac12(u-\sigma_r)(0)=-\frac{M_r}{4}, 
 \end{equation}  that is  \begin{equation} u-\sigma_r\leq-M_r/4 \text{ over } \frac{1}{2}\omega_r\label{eq:u_sig} \end{equation}
Set $0<r<\widetilde{\beta}/2$ and $x'\in A_{\gamma/2}$. We apply the inequality \begin{equation}\nonumber\text{For } t_0\leq t_1 \text{ and }f\in H^1([t_0,t_1]),  \ \int_{t_0}^{t_1}f'^2(t)dt\geq \frac{(f(t_1)-f(t_0))^2}{t_1-t_0},\end{equation} to $(u-\sigma_r)(\cdot,x')$ over the segments $[a(x'),\alpha(x')]$ and $[\beta(x'),b(x')]$, each of which has length smaller than $2r$ (see \eqref{eq:omegar}), obtaining \begin{align}\nonumber\int_{a(x')}^{b(x')}(\partial_{q_r}(u-\sigma_{r}))^2dt&\geq \int_{a(x')}^{\alpha(x')}(\partial_{q_r}(u-\sigma_{r}))^2dt+\int_{\beta(x')}^{b(x')}(\partial_{q_r}(u-\sigma_{r}))^2dt\\&\geq \frac{[u(\alpha(x'),x')-\sigma_r(\alpha(x'),x')]^2}{\alpha(x')-a(x')}+\frac{[u(\beta(x'),x')-\sigma_r(\beta(x'),x')]^2}{b(x')-\beta(x')}\geq \frac{M_{r}^2}{16r}\nonumber\end{align}  since $(u-\sigma_r)(a(x'),x')=(u-\sigma_r)(b(x'),x')=0$, and where we also used \eqref{eq:u_sig} in the last inequality.

We write $|\nabla (u-\sigma_r)|\geq |\partial_{q_r} (u-\sigma_r)|$ and integrate the above estimate over $x'\in A_{\gamma/2}$; this yields \begin{align}\nonumber\int_{\omega_r\cap B_{\widetilde{\beta}}(0)}|\nabla (u-\sigma_{r})|^2d\H^{N-1}&\geq \int_{A_{\gamma/2}}\left(\int_{a(x')}^{b(x')}(\partial_{q_r}(u-\sigma_{r}))^2dt\right)d\H^{N-2}(x')\\ \nonumber&\geq \frac{M_{r}^2\H^{N-2}(A_{\gamma/2})}{16r}\\ \nonumber&\geq \delta\frac{M_{r}^2\H^{N-2}(\widetilde{\omega_{r}})}{r}\end{align} for some constant $\delta=\delta(\widetilde{\beta},\text{diam}(\Om))>0$, getting \eqref{eq:shapebelowperim}.\\ 

\noindent  {\bf Conclusion of the estimate from below: }Plugging \eqref{eq:shapebelowperim} and \eqref{eq:lineartermvol} into \eqref{eq:belowalmostfinal} gives \begin{equation} \label{eq:shapefinalperim}\forall r\in (0,\widetilde{\beta}/2), \ \frac{M_r^2}{r}\H^{N-2}(\widetilde{\omega_r})\leq C\big(P(K)-P(K_r)+M_rr\H^{N-2}(\widetilde{\omega_r})\big) \end{equation} where $C=C(N,d(x_0,\partial \Om), c(\beta), \|\nabla u\|_{L^{\infty}(B_{\beta})},\text{diam}(\Om))$, which completes the proof of the estimate from below.

\begin{center}{\bf Conclusion}\end{center} 
We claim that for all $r\in (0,\widetilde{\beta})$ it holds\begin{align}\H^{N-1}(\{u\leq\sigma_r\})&=\H^{N-1}(\omega_r)+\H^{N-1}(\{u\leq\sigma_r\}\setminus\omega_r)\nonumber\\&\leq (1+C)\H^{N-1}(\omega_r)\label{eq:kappar}\end{align} with $C=\gamma^{1-N}$, where $\gamma=\frac{\widetilde{\beta}}{2\diam(\Om)}$ was introduced to obtain \eqref{eq:incl_gamma}. Indeed, denoting by $\kappa_r:=\{u\leq\sigma_r\}\setminus\omega_r$, it suffices to show \begin{equation}\label{eq:gammaab}\forall r\in (0,\widetilde{\beta}), \ \gamma\kappa_r\subset \omega_r\end{equation}
Let $x\in \kappa_r$; we have $(u-\sigma_r)(\gamma x)\leq 0$ using the convexity of $u-\sigma_r$, and furthermore $\gamma x\in B_{\widetilde{\beta}}(0)\subset \text{Int}(\Om_r)$ (see \eqref{eq:BbetaOmr}). This provides $\gamma x\in \omega_r$, allowing to conclude that \eqref{eq:gammaab} holds.\\\\ Injecting \eqref{eq:kappar} into \eqref{eq:finalvol} and using \eqref{eq:upomegar} we get that there is a constant $C=C(\widetilde{\beta},\diam(\Om))$ such that \begin{equation} \forall r\in (0,\widetilde{\beta}), \ |K\setminus K_r|\leq CM_rr\H^{N-2}(\widetilde{\omega_r})\label{eq:finalvol2}\end{equation} 
Gathering \eqref{eq:shapefinalperim} and \eqref{eq:finalvol2} and recalling \eqref{eq:shape_variation} finally provides the existence of $r_0=r_0(\widetilde{\beta},\text{diam}(\Om),\eps_K, \|\nabla u\|_{L^{\infty}(B_{\beta})})$ such that \begin{equation}\forall r\in(0,r_0),\ \frac{M_r^2}{r}\H^{N-2}(\widetilde{\omega_r})\leq CM_rr\H^{N-2}(\widetilde{\omega_r})\label{eq:finalshape}\end{equation} where $C=C(N, \Lambda_K, d(x_0,\partial \Om), c(\beta), \|\nabla u\|_{L^{\infty}(B_{\beta})},\text{diam}(\Om))$. Thanks to \eqref{eq:downomegar} we can simplify by $\H^{N-2}(\widetilde{\omega_r})$ in \eqref{eq:finalshape}, to get that $M_r\leq Cr^2$. This completes the proof. 
\qed

\begin{remark}\label{th:rkHolder1} \begin{itemize}\item   If $N=2$, the proof can be simplified, as one can show that there exists some $r_0$ such that $\Om_r=\Om$ for $r< r_0$. Indeed, since $H$ is one-dimensional, the right inclusion of \eqref{eq:omegar} reads $\{x\in \Om, u(x)\leq \sigma_r(x)\}\subset [y-r,y+r]$, which shows that $\Om_r= \Om$ for small $r$ using \eqref{eq:cbeta}.

On the other hand, for $N\geq3$ it is not hard to find convex bodies $K$ such that $\Om_r\subsetneq \Om$ for each (small) $r$. For instance, let us consider $K:=C\cap B$ the intersection of the unit ball $B\subset\R^N$ with the cylinder $C:=[-1,1]\times B_{\R^{N-1}}((0,\ldots,0,1),1)$. Although it is possible to compute explicitly $\Om$ let us just notice that $\Om\supset (-1,1)\times\{0\}^{N-2}=:S$. We consider the situation where $x_0=y=0$ with $p=0\in \partial u(x_0)$, so that $l\equiv0$ in this case. Then we see that $q_r=(0,1,0,\ldots,0)$ satisfies $M_r=\sup_{B_r(0)}u=u(rq_{r})$. Note that $u\equiv 0$ along $S$, while on the other hand $\sigma_r= M_r/2>0$ over $S$. As a consequence, since $K\cap P=B\cap P$ where $P$ denotes the $(x_1,x_N)$ plan, we find $x=(x_1,0,\ldots,0)\in S$ close enough from $(-1,0,\ldots,0)$ such that $\widehat{v_r}(x)=\widehat{\sigma_r}(x)\notin K$, thus getting that $\Om_r\subsetneq \Om$ in this case.
 \item Using the same ideas as in \cite{CCL} where they study the regularity of $u$ solution of \eqref{eq:model} when $f\in L^p(\Om)$ (for some $p>N$), one can prove the $C^{1,\alpha}$ regularity of a convex body $K$ satisfying 
\begin{equation}\label{eq:qmpcgamma} \forall \widetilde{K}\in \K^N\textrm{ such that } \widetilde{K}\subset K \textrm{ and } |K\setminus\widetilde{K}|\leq\eps,\ \ P(K)\leq P(\widetilde{K})+\Lambda|K\setminus\widetilde{K}|^\gamma\end{equation}
for some $\eps>0, \Lambda\geq 0$ and $\gamma\in(1-1/N,1]$. In this case, instead of \eqref{eq:finalshape}, we derive with the same arguments \begin{equation}\nonumber \frac{M_r^2}{r}\H^{N-2}(\widetilde{\omega_r})\leq C\big(|K\Delta K_r|^{\gamma}+ M_rr\H^{N-2}(\widetilde{\omega_r})\big)\leq CM_r^{\gamma}r^{\gamma}\H^{N-2}(\widetilde{\omega_r})^{\gamma}\end{equation}Using \eqref{eq:downomegar} we get \begin{equation}M_r^{2-\gamma}\leq Cr^{(N-2)(\gamma-1)+\gamma +1}\end{equation} Direct computation gives that $(N-2)(\gamma-1)+\gamma+1>2-\gamma$ whenever $\gamma\in(1-1/N,1]$, so that the classical result \cite[Lemma 3.1]{DF} gives that $K$ is $C^{1,\alpha}$ with $\alpha=(N(\gamma-1)+1)/(2-\gamma)$. \end{itemize}\end{remark}

\subsection{Regularity with volume constraint}\label{sect:shape3}
We now focus on problems having a volume constraint, as they often appear in applications. We thus consider the problem \begin{equation} \label{eq:pbshapeconstvol} \inf\left\{P(K)+R(K),\ K\in \K^N,  K\subset D, |K|=V_0\right\}\end{equation} for some convex body $D\in \K^N$ and $0< V_0< |D|$.  Existence for this problem can be shown under the assumption \eqref{eq:hypR4} below made upon $R$ (see Theorem \ref{th:general_existence} (i)). In this section we prove that under suitable assumptions on $R$, minimizers of this problem are $C^{1,1}$. We use a penalization method to prove that these solutions are quasi-minimizer of the perimeter under convexity constraint.

 \subsubsection{Preliminaries}
 Before introducing the hypothesis which we will make upon $R$, let us recall the notion of Hausdorff distance between sets. If $A$ and $B$ are non-empty compact subsets of $\R^N$, the Hausdorff distance $d_H(A,B)$ between $A$ and $B$ is defined as the quantity \begin{equation}\nonumber d_H(A,B):=\max\left\{\sup_{x\in A}d(x,B),\sup_{x\in B}d(A,x)\right\}\end{equation} where $d(\cdot,\cdot)$ denotes the euclidean distance. The Hausdorff distance $d_{H}$ is a distance over the class of non-empty compact sets of $\R^N$. 

Let us recall two classical facts about $d_{H}$, whose proof is given in the Appendix:
\begin{prop}\label{prop:hausdorff}
Let $D\in\K^N$, $(K_n)$ be a sequence of convex bodies verifying $K_n\subset D$ for any $n\geq0$, and let $K\subset D$ be a non-empty compact convex set. Then 
\begin{enumerate}
\item We have the equivalence: \begin{equation} \label{eq:equivHausVol} d_H(K_n,K)\rightarrow0 \Longleftrightarrow |K_n\Delta K|\rightarrow0\end{equation}
\item If $d_H(K_n,K)\rightarrow0$, and $C\in\K^N$ is such that $C\subset \text{Int}(K)$ then \begin{equation}\label{eq:opensubH} C\subset K_n\;\; \text{ for large }\; n.\end{equation}
\end{enumerate}\end{prop}

We now introduce a new assumption on $R$ which is slightly stronger than \eqref{eq:hypR3}, as seen in Proposition \ref{th:stronghypR} below: for any $D'\in \K^N$  with $D'\subset D$ we set \begin{equation}\label{eq:defKD'D}\K^N_{D',D}:=\{K\in \K^N, \ D'\subset K\subset D\}\end{equation} 
and we assume \begin{equation}\label{eq:hypR4} \forall D'\subset D\in \K^N, \ \exists C_{D',D}>0, \ \forall \big( K_1,K_2\in \K^N_{D',D}, \ K_1\subset K_2\big), \ \ |R(K_2)-R(K_1)|\leq C_{D',D}|K_2\setminus K_1|\end{equation}  
 \begin{prop}\label{th:stronghypR} If $R$ satifies \eqref{eq:hypR4} then $R$ satisfies \eqref{eq:hypR3}.\end{prop}
\begin{proof} Letting $K\in\K^N$ be fixed,  there exists $D'\in\K^N$ such that $D'\subset \text{Int}(K)$. Thanks to \eqref{eq:opensubH} we know that there exists $\delta>0$ such that if $\widetilde{K}\in\K^N$ verifies $d_H(K,\widetilde{K})\leq \delta$ then $\widetilde{K}\supset D'$. Thanks to \eqref{eq:equivHausVol} we can find $\eps_K>0$ such that if $|K\setminus\widetilde{K}|\leq\eps_K$, $d_H(K,\widetilde{K})\leq \delta$. Putting these together and applying hypothesis \eqref{eq:hypR4} with the class $\K_{D',K}^N$ gives that $R$ satisfies \eqref{eq:hypR3}.\end{proof}

Note that on the other hand condition \eqref{eq:hypR4} is genuinely stronger than \eqref{eq:hypR3}. In fact, \eqref{eq:hypR4} is double-sided while it is not the case for \eqref{eq:hypR3}, but there is a deeper difference which boils down to the fact that in \eqref{eq:hypR3} the constants $(\eps_K,C_K)$ depend on $K$, while in \eqref{eq:hypR4} the constant $C_{D',D}$ is locally uniform. In this sense, \eqref{eq:hypR3} somehow says that $R$ is "differentiable" everywhere while \eqref{eq:hypR4} means that $R$ is locally Lipschitz; one can build an example of $R$ verifying a double-sided \eqref{eq:hypR3} and not \eqref{eq:hypR4} by setting $R(K):=f(|K|)$ for some $f:\R^+\rightarrow\R$ differentiable everywhere but not locally Lipschitz.


 \subsubsection{Main result}
The main result of this section is the following.
\begin{theorem}\label{th:shape3}
Let $K^{*}$ be a solution of problem \eqref{eq:pbshapeconstvol}, with $R$ satisfying \eqref{eq:hypR4} and $0<V_0<|D|$. 
Then $K^*$ is a quasi-minimizer of the perimeter in the sense of Definition \ref{def:qmpc}, and is therefore $C^{1,1}$. \end{theorem}

The proof of Theorem \ref{th:shape3} relies on the following important lemma, which allows the use of the results of section \ref{sect:shape1} over an auxiliary problem for which $K^*$ is still optimal.
For any $K\in \K^N$ and $\eps>0$ we set the class $\mathcal{O}_{\eps}(K)$ of convex bodies which are $\eps-$close perturbations of $K$ from the inside: \begin{equation} \nonumber \mathcal{O}_{\eps}(K):=\{\widetilde{K}\in\K^N, \ \widetilde{K}\subset K,\ |K\setminus\widetilde{K}|\leq \eps\}.\end{equation} 
\begin{lemma}\label{lem:shapepenal}
Let $K^{*}$ be a solution of problem \eqref{eq:pbshapeconstvol}, with $R$ satisfying \eqref{eq:hypR4} and $0<V_0<|D|$. Then there exists $\Lambda>0$ and $\eps>0$ such that $K^{*}$ is a solution of 
\begin{equation} \label{eq:pbshapeconstpenal}\min\left\{P(K)+R(K)+\Lambda\big||K|-V_0\big|,\ K\in \mathcal{O}_{\eps}(K^*)\right\}\end{equation}
\end{lemma}

 We will use in the proof of this lemma the following classical result concerning Minkowski sums and mixed volume (see for instance \cite[Theorem 5.1.7]{Sc}). \begin{theorem}[Mixed volume]\label{th:mixed} For any $m\in \N^*$ and $K_1,...,K_m\in \K^N$, the map $(t_1,...,t_m)\in (\R^+)^m\mapsto |t_1K_1+...+t_mK_m|$ is a homogeneous polynomial of degree $N$, \textit{i.e.} there exists a symmetric function $V:(\K^N)^N\rightarrow\R$ (called mixed volume) such that for any $t_1,t_2,...,t_m\geq0$ \begin{equation}\nonumber |t_1K_1+...+t_mK_m|=\sum_{i_1,...,i_N=1}^mt_{i_1}...t_{i_N}V(K_{i_1},...K_{i_N})\end{equation}Furthermore $V$ is nondecreasing in each coordinate for the inclusion of sets, and continuous for the Hausdorff distance. \end{theorem}

\noindent{\bf Proof of Lemma \ref{lem:shapepenal}: }Set $G(K):=P(K)+R(K)$ for any $K\in \K^N$. We use a classical strategy (see for example \cite[Lemma 4.5]{DLPV}, though our construction will be adapted to the convexity constraint): for any $K\subset K^*$ with $|K^*\setminus K|\leq \eps$ (for sufficiently small $\eps$) we build a convex body $\widetilde{K}\subset D$  such that $|\widetilde{K}|=V_0$ and $G(\widetilde{K})\leq G(K)+\Lambda\big||K|-V_0\big|$ (for sufficiently large $\Lambda$). Writing then \begin{equation}\nonumber G(K^*)\leq G(\widetilde{K})\leq G(K)+\Lambda\big||K|-V_0\big|\end{equation} yields the conclusion.

For $K\subset D$ a convex body and $t\in [0,1]$ we set the Minkowski sum \begin{equation}\nonumber K_t:=(1-t)K+tD\end{equation}and note that  $K_t$ is a convex body and $K_t\subset D$. We first claim that there exist $\eps_{0}>0, c>0, t_0>0$ such that
\begin{equation}\label{eq:below_vol_derivative}\forall K\in \mathcal{O}_{\eps_{0}}(K^*), \ \forall t\in [0,t_0],\ |K_t|-|K|\geq ct\end{equation} 
Let $f_K(t):=|K_t|$. By Theorem \ref{th:mixed}, $f_K$ is polynomial with degree $N$ and more precisely
$$f_{K}(t)=\sum_{k=0}^N\binom{N}{k}(1-t)^kt^{N-k}V(K[k],D[N-k])$$
where $K[p]$ stands for $(K,K,\cdots,K)$ with $p$ repetitions.

Now, as the class $\{L\subset \R^N \text{ compact convex},\ L\subset D\}$ is compact for $d_H$ and since $V$ is continuous for $d_H$, we deduce that the coefficients of $f_K$ are uniformly bounded for $K\in \K^N, K\subset D$. To conclude that the claim holds it therefore suffices to show that $f_K'(0)$ is bounded from below by a positive constant uniform in $K\in \mathcal{O}_{\eps}(K^*)$ for some small $\eps$.

One has 
\begin{align}\nonumber f_K'(0)& =N\big(V(K[N-1],D)-V(K[N])\big)\\ &=N\big(V(K,...,K,D)-|K|\big)\nonumber
\end{align}
which is nonnegative by monotonicity of the mixed volume. Moreover, as soon as we have $K\subsetneq D$, we can apply \cite[Theorem 7.6.17]{Sc} to get ${V(K[N-1],D)-|K|>0}$; equality would in fact imply that $D$ is a $0$-tangential body of $K$, hence that $K=D$. This gives in particular $f_{K^*}'(0)>0$ (since $V_{0}<|D|$).  Since $K\mapsto f_K'(0)$ is continuous for $d_H$, we therefore have $f_K'(0)\geq f_{K^*}'(0)/2$ for any convex body $K\subset D$ with $d_H(K,K^*)$ small enough. Thanks to \eqref{eq:equivHausVol}, we deduce the existence of $\eps_{0}>0$ such that $f_K'(0)\geq f_{K^*}'(0)/2>0$ for any $K\in \mathcal{O}_{\eps_{0}}(K^*)$. This yields \eqref{eq:below_vol_derivative} for $c:=f_{K^*}'(0)/4$ and for some small $t_0=t_0(D,K^*)$.

We now show that a reverse inequality holds for the perimeter: there exists $C=C(D)>0$ such that \begin{equation}\label{eq:above_per_derivative}\ \forall t\in [0,1], \ \forall K\in \K^N \textrm{ such that }K\subset D, \ P(K_t)-P(K)\leq Ct\end{equation} 
 Since for any $L\in \K^N$, \begin{equation}\label{eq:mixedper} P(L)=NV(L[N-1],B)\end{equation} where $B$ is the ball of unit radius (see for instance \cite[p.294, (5.43) to (5.45)]{Sc}) the mapping $t\mapsto P(K_t)=NV(K_t[N-1],B)$ is a polynomial function whose coefficients are continuous quantities of  ${(V(K[i],D[N-1-i],B))_{0\leq i\leq N-1}}$. Hence the continuity of $V$ for $d_H$ and the compactness of the class $\{L\subset \R^N \text{ compact convex},\ L\subset D\}$ for $d_H$ give \eqref{eq:above_per_derivative}.
 
 Putting together \eqref{eq:above_per_derivative} with \eqref{eq:below_vol_derivative} and setting $C':=C/c$ provides 
 \begin{equation} \label{eq:ab_per_vol}
 \forall K\in \mathcal{O}_{\eps_{0}}(K^*), \ \forall t\in [0,t_0],\ P(K_t)-P(K)\leq C' \big||K_t|-|K|\big|\end{equation}
On the other hand,  there exists $D'\in \K^N$ such that $D'\subset \text{Int}(K)$.  Then arguing as in the proof of Proposition \ref{th:stronghypR} ensures that for $\eps_{1}$ small enough, any $K\in\mathcal{O}_{\eps_{1}}(K^*)$ verifies $K\supset D'$. Therefore by \eqref{eq:hypR4} there exists $C_{D',D}>0$ such that
 \begin{equation}\label{eq:above_R}\forall t\in [0,1], \  \forall K\in \mathcal{O}_{\eps_{1}}(K^*), \ R(K_t)-R(K)\leq C_{D',D}|K_t\setminus K| = C_{D',D}\big||K_t|-|K|\big|. \end{equation}
Let $\eps=\min\{\eps_{0},\eps_{1},ct_{0}\}$ and $K\in\O_{\eps}(K^*)$.

We deduce from \eqref{eq:below_vol_derivative} that
  \begin{equation}\nonumber |K_{t_0}|-|K^*|{=|K_{t_0}|-|K|+|K|-|K^*|\geq ct_0-\eps}\geq 0\end{equation}
   By continuity of $t\mapsto|K_{t}|$, this ensures that 
   
   there exists $t\in [0,t_0]$ such that $|K_t|=|K^*|$. With \eqref{eq:ab_per_vol} and \eqref{eq:above_R} we get that the set $\widetilde{K}:=K_t$ satisfies all the requirements laid out at the beginning of the proof with $\Lambda:=C'+C_{D',D}$.
\qed

\noindent{\bf Proof of Theorem \ref{th:shape3}:} Thanks to Lemma \ref{lem:shapepenal}, an optimal shape $K^*$ for \eqref{eq:pbshapeconstvol} is solution of \eqref{eq:pbshapeconstpenal} for some $\eps>0$ and $\Lambda>0$. 
Therefore we have
$$\forall K\in\O_{\eps}(K^*), \;\;P(K^*)\leq P(K)+R(K)-R(K^*)+\Lambda|K^*\setminus K|,$$
 By Proposition \ref{th:stronghypR}, $R$ verifies hypothesis \eqref{eq:hypR3}, and as a consequence there exists $C_{K^*}>0$ such that  \[\forall K\in\O_{\eps}(K^*), \;\;P(K^*)\leq P(K)+(C_{K^*}+\Lambda)|K^*\setminus K|\] Hence $K^*$ is a quasi-minimizer of the perimeter under convexity constraint in the sense of Definition \ref{def:qmpc}. We can therefore apply Theorem \ref{th:shape1} to get that $K^*$ is $C^{1,1}$. \qed

\begin{remark}\label{th:rkHold2}As in Remark \ref{th:rkHolder1}, there is an analogous result to Theorem \ref{th:shape3} if $R$ is merely $\gamma$-H\"{o}lder for some $\gamma \in(1-1/N,1]$, meaning that  \eqref{eq:hypR4} is replaced with \begin{equation}\label{eq:hypRHold2} \forall D'\subset D\in \K^N, \ \exists C_{D',D}>0, \ \forall \big( K_1,K_2\in \K^N_{D',D}, \ K_1\subset K_2\big), \ \ |R(K_2)-R(K_1)|\leq C_{D',D}|K_2\setminus K_1|^{\gamma}\end{equation}In this case, keeping the same notations as in the proof of Lemma \ref{lem:shapepenal}, the same arguments show \begin{equation}\nonumber G(\widetilde{K})\leq G(K)+\Lambda\big||K|-V_0\big|^{\gamma},\end{equation} 
and with the additionnal remark that 
$$P(K_t)-P(K)\leq \Lambda \big||K_t|-|K|\big|\leq \Lambda'\big||K_t|-|K|\big|^{\gamma}$$
 with $\Lambda':=\Lambda\times |D|^{1-\gamma},$
we conclude in this case that the optimal shape is $C^{1,\alpha}$ for the same $\alpha$ as in Remark \ref{th:rkHolder1}.
 \end{remark}

\section{Examples and applications}\label{sect:examples}
This section is dedicated to applications of the results of Section \ref{sect:shape3}. We therefore provide examples of functionals $R$ satisfying hypothesis \eqref{eq:hypR4} (and therefore \eqref{eq:hypR3} as well), Theorem \ref{th:shape3} then implying that the minimizers of the corresponding problem are $C^{1,1}$-regular.

\subsection{First examples}

Let us start by giving  two examples taken from the literature of minimization of $P+R$ for which proving that the functional $R$ satisfies hypotheses \eqref{eq:hypR3} and \eqref{eq:hypR4} is quite easy. 

A first example of relevant $R$ is given through the following model of a liquid drop subject to the action of a potential:  it consists in minimizing the energy \[P(E)+\int_Eg\] among bounded subsets $E$ of $\R^N$  with given volume, where $g:\R^N\rightarrow\R$ is a fixed function in  $L^1_{\text{loc}}(\R^N)$, see for example \cite{FM11Ont}.  An optimal shape may not be convex, and in this case it is interesting to study the counterpart of this problem with an additional convex constraint: under reasonable hypotheses on $g$ we can prove existence and regularity for a minimizer of this problem in the class of convex shapes (see Proposition \ref{th:1stex} below).

We can also consider the following generalization of the Gamow model \eqref{eq:gamow}, which consists in the minimization \[\inf\left\{P(E)+\int_E\int_E\frac{dxdy}{|x-y|^{N-\alpha}},\ E\subset \R^N,\ |E|=V_0\right\}\] for a given mass $V_0\in(0,+\infty)$ and parameter $\alpha\in(0,N)$. The interaction term $V_\alpha(E):=\int_E\int_E|x-y|^{\alpha-N}dxdy$, called {\it Riesz potential}, is maximized by any ball of volume $m$ by virtue of the Riesz inequality, so that there is a competition in the above minimization. As in the case of \eqref{eq:gamow} it is known that for small masses $m$ the ball of corresponding volume is the unique (up to translation) solution to the minimization problem (see the Introduction of \cite{FFMMM} for a review of these results), while it is proven in \cite[Theorem 2.5]{KM1} and \cite[Theorem 3.3]{KM2} for $\alpha\in(N-2,N)$ that beyond a certain threshold of mass $m$ there is no existence for this problem (and for $\alpha=N-2$ in \cite[Theorem 3]{FN}). On the contrary, the convexity constraint will enforce existence for all masses ; furthermore, applying Theorem \ref{th:shape3} we are able to show regularity for the problem when considered under a convexity constraint, see Proposition \ref{th:1stex} below.

We thus have the following proposition.
\begin{prop}\label{th:1stex}Let $N\geq2$ and $V_0\in(0,+\infty)$.\\ Let $g\in L^{\infty}_{\text{loc}}(\R^N)$ be coercive, that is to say $\lim_{|x|\to\infty}g(x)=+\infty$. Then there exists a solution to the problem \[\inf\left\{P(K)+\int_Kg,\ K\in\K^N,\ |K|=V_0\right\}\] and any such solution $K^*$is $C^{1,1}$. \\ Let $\alpha\in(0,N)$. Then there exists a solution to the problem \[\inf\left\{P(K)+V_{\alpha}(K),\ K\in \K^N, \ |K|=V_0\right\}\] and any such solution $K^*$ is $C^{1,1}$.\end{prop}

 \begin{proof}
\noindent{\bf Existence:} \begin{enumerate} \item  Let $(K_n)$ be a minimizing sequence for the first problem. As $g$ is coercive, there exists a bounded set $A$ such that $g\geq 0$ outside $A$. Therefore, as there exists $C>0$  such that $P(K_n)+\int_{K_n}g\leq C$ by definition of $(K_n)$, we can write \begin{equation*} P(K_n)\leq C-\int_{K_n}g\leq C -\int_{K_n\cap A}g\leq C+\|g\|_{L^1(A)}\end{equation*}so the perimeters $P(K_n)$ are uniformly bounded. We thus use the inequality \begin{equation}\label{eq:diamPV} \text{diam}(K)\leq C(N)\frac{P(K)^{N-1}}{|K|^{N-2}} \end{equation} valid for any convex body $K$ (see \cite[Lemma 4.1]{EFT}) to get that the $\text{diam}(K_n)$ are also uniformly bounded, recalling also that $|K_n|=V_0$. As a consequence, using the coercivity of $g$ we now show that there is no loss of generality in assuming that the $K_n$ are uniformly bounded: let $r>0$ be such that $B_r$ the ball centered at $0$ of radius $r$ verifies $\text{diam}(K_n)\leq \text{diam}(B_r)$ for every $n$, and set $m:=\text{ess}\sup_{B_r}g$. Thanks to the fact that $g$ is coercive we can find $r'>r$ such that $g\geq m$ outside $B_{r'}$. For any fixed $n\in\N$, either $K_n\subset B_{2r'}$ and we set $\widetilde{K_n}:=K_n$, or else there exists $x\in K_n\cap (B_{2r'})^c$ so that $K_n\subset (B_{r'})^c$ thanks to the bound on the diameters. In this latter case, we thus have $g\geq m $ over $K_n$ while $K_n-x\subset B_{r'}$ gives that $g\leq m $ over $K_n-x$ , so that \[P(K_n-x)+\int_{K_n-x}g\leq P(K_n)+m|K_n|\leq P(K_n)+\int_{K_n}g\] using also the translation invariance of the perimeter. We then rather set $\widetilde{K_n}:=K_n-x$. This argument ensures that the sequence $(\widetilde{K_n})$ is still minimizing with the additionnal property that $\widetilde{K_n}\subset B_{2r'}$ for each $n$. We will keep denoting it $K_n$.

Now, $K\mapsto \int_Kg$ satisfies \eqref{eq:hypR4}: if $D\in\K^N$ and $K'\subset K\subset  D$ are convex bodies, then \begin{equation}\label{eq:intg_R}\left|\int_Kg-\int_{K'}g\right|\leq\|g\|_{L^\infty(D)}|K\setminus K'|.\end{equation}
Using the Blaschke selection theorem and \eqref{eq:equivHausVol} we can extract a subsequence (still denoted $(K_n)$) and a compact convex $K^*$ such that $K_n\rightarrow K^*$ for the Hausdorff distance and in volume. In particular $|K^*|=V_0$. Thanks to \eqref{eq:intg_R} we have that $\int_{K_n}g\rightarrow \int_{K^*}g$ and $P(K_n)\rightarrow P(K^*)$ by continuity of the perimeter for convex domains (see for instance \cite[Proposition 2.4.3, (ii)]{BB}). We thus get existence.
\item 
Let $(K_n)$ be a minimizing sequence for the second problem. Since $V_\alpha$ is nonnegative we immediately get that $(P(K_n))$ is bounded, getting thus from \eqref{eq:diamPV} that the sequence $\diam(K_n)$ is bounded as well. By translation invariance of the perimeter and of $V_\alpha$ there is not loss of generality in assuming that there exists a compact set $D$ such that $K_n\subset D$ for each $n$. 

Let us now show that $V_\alpha$ verifies \eqref{eq:hypR4}. This was done in \cite[Equation (2.11)]{KM2}, but we reproduce hereafter the short argument for the convenience of the reader. Let $D\in\K^N$ and $K'\subset K\subset D$ be convex bodies. We set $v_E(x):=\int_E|x-y|^{\alpha-N}dy$ for any compact set $E$ and write $f(x,y):=|x-y|^{\alpha-N}$. We have \begin{align}\nonumber 0\leq V_\alpha(K)-V_\alpha(K')&=\int_K\int_Kf-\int_K\int_{K'}f+\int_{K\setminus K'}\int_{K'}f\\\nonumber&= \int_{K\setminus K'}v_K+v_{K'}\\ &\leq 2|K\setminus K'|\left(\int_{B_D}\frac{dy}{|y|^{N-\alpha}}\right)\label{eq:RV_alpha}\end{align} with $B_D$ a ball of volume $|D|$, where we used that \[v_{K'}(x)\leq v_K(x)=\int_K\frac{dy}{|x-y|^{N-\alpha}}=\int_{x-K}\frac{dy}{|y|^{N-\alpha}}\leq \int_{B_D}\frac{dy}{|y|^{N-\alpha}}\] thanks to the Hardy-Littlewood inequality and since $|K|\leq |D|$.

Since $V_\alpha$ verifies \eqref{eq:hypR4} and $K_n\subset D$ for all $n\in \N$ we conclude to existence as before using the Blaschke selection theorem.
\end{enumerate}

\noindent{\bf Regularity:} We proved respectively in \eqref{eq:intg_R} and \eqref{eq:RV_alpha} that $K\mapsto \int_Kg$ and $V_\alpha$ satisfy \eqref{eq:hypR4}. We can therefore apply Theorem \ref{th:shape3} to get that any minimizer is $C^{1,1}$. 
\end{proof}

\subsection{PDE and Spectral examples}\label{sect:pde}

We now focus on more difficult examples, which will lead to the proof of Theorem \ref{th:intro} given in the introduction.
Let us first set some notations and definitions.
 
  If $\Om\subset \R^N$ is a bounded Lipschitz open set we denote respectively by 
 \begin{align} \nonumber 0<\lambda_1(\Om)\leq\lambda_2(\Om)\leq\cdot\cdot\cdot \leq\lambda_n(\Om)\leq\cdot\cdot\cdot\nearrow+\infty\\ 0=\mu_1(\Om)\leq\mu_2(\Om)\leq\cdot\cdot\cdot\leq\mu_n(\Om)\leq \cdot\cdot\cdot\nearrow+\infty\nonumber\end{align}
  the nondecreasing sequence of the Dirichlet and Neumann Laplacian eigenvalues associated to $\Om$ (see for example \cite{H20Sha} for more details). We also define $\tau(\Om)$ the torsional rigidity of $\Om$ as
  $$\tau(\Om)=\int_{\Om}u_{\Om}dx=-2\min\left\{\int_{\Om}\frac{|\nabla u|^2}{2}- \int_{\Om}fu,\ u\in H^1_0(\Om)\right\}$$
  where $u_{\Om}$ is the unique solution of 
  \begin{equation}\label{eq:torsion}\begin{cases}-\Delta u=1 \text{ in } \Om\\ u\in H^1_0(\Om)\end{cases}\end{equation}
 
For any convex body $K\in \K^N$ we will frequently use the notation $\Om_K:=\text{Int}(K)$, and then define $\lambda_n(K):=\lambda_n(\Om_K)$,  $\mu_n(K):=\mu_n(\Om_K)$ for $n\in \N^*$, and $\tau(K):=\tau(\Om_K)$.

We are now ready to state the main result of this section, which will be proved later on in Sections \ref{sect:Dirichlet} and \ref{sect:Neumann}.

\begin{theorem}\label{th:examples} 
 Let $n\in \N^*$, $N\geq2$. Then any $R\in\{\lambda_{n}, \mu_{n}, \tau\}$ satisfy \eqref{eq:hypR3} and \eqref{eq:hypR4}, namely 
for every  $D'\subset D \subset \R^N$ convex bodies there exists $C=C(D',D,R)$ such that for any $K'\subset K$ lying in $\K^N_{D',D}$  (defined in \eqref {eq:defKD'D})\begin{equation*} \left|R(K)-R(K')\right|\leq C|K\setminus K'|.\end{equation*}
\end{theorem}

\begin{remark}\label{rk:improvment}The fact that $K'\subset K$ is not essential to ensure that Lipschitz estimates hold. In fact, one has that \[\exists C>0, \forall (K,K')\in \K^N_{D',D}, \;\; |R(K)-R(K')|\leq C|K\Delta K'|\]
 by applying Theorem \ref{th:examples} with $K$ and $K\cup K'$ on the one hand, $K'$ and $K\cup K'$ on the other hand.\end{remark}

As a consequence, combined with  Corollary \ref{th:shapecor} and Theorem \ref{th:shape3}, we are able to prove Theorem \ref{th:intro}. 

\noindent{\bf Proof of Theorem \ref{th:intro}:} Recall that $R(K):=F(|K|, \tau(K), \lambda_{1}(K),\ldots,\lambda_{n}(K),\mu_{1}(K),\ldots,\mu_{n}(K))$ for some $F:(0,+\infty) \times (0,+\infty)\times (0,+\infty)^n\times\R_{+}^n \rightarrow\R$ locally Lipschitz. Let us show that $R$ satifies \eqref{eq:hypR4}, so that it also satisifes \eqref{eq:hypR3} (thanks to Proposition \ref{th:stronghypR}) and Corollary \ref{th:shapecor} and Theorem \ref{th:shape3} give the results. 

Let $D_1\subset D_2\in\R^N$ be convex bodies and let $K,K'\in \K^N_{D_1,D_2}$ with $K'\subset K$. Set $L=K$ or $L=K'$. Then from monotonicity of Dirichlet eigenvalues and torsion, for any $k\in \N^*$ it holds \begin{align}\nonumber \lambda_k(D_2)\leq \lambda_k(L)\leq \lambda_k(D_1) \\ \tau(D_1)\leq \tau(L)\leq \tau(D_2)\nonumber\end{align} Moreover, since $\mu_k(L)\leq\lambda_k(L)\leq\lambda_k(D_1)$ we have for any $k\in \N^*$  \begin{equation}\nonumber \mu_k(L)\leq\lambda_k(D_1)\end{equation}Also, \[|D_1|\leq |L|\leq |D_2|\]Putting these four estimates together and  using that $F$ is locally Lipschitz we find $C(F,D_1,D_2)$ such that \begin{eqnarray*}\nonumber|R(K)-R(K')|&\leq& C(F,D',D)\big(\sum_{k=1}^n|\lambda_k(K)-\lambda_k(K')|\\&&+\sum_{k=1}^n|\mu_k(K)-\mu_k(K')|+|\tau(K)-\tau(K')|+\left||K|-|K'|\right|\big) \end{eqnarray*}Applying Theorem \ref{th:examples} for $\lambda_k, \mu_k$ and $\tau$ and noticing that $\left||K|-|K'|\right|=|K\setminus K'|$ ensures that $R$ satisfies \eqref{eq:hypR4}. The result follows.\qed
  
\subsubsection{A general existence result}
In this short section we show a general existence result for the minimization among convex sets of a functionnal of the type $P+R$, where $R$ is mostly thought of as a PDE-type functional. Using mild continuity of $R$, we show existence of a minimizer under additional box and volume constraints, and we also show existence in the unconstrained case with coercivity assumptions of $R$.
The statement is as follows.

\begin{theorem}\label{th:general_existence}(i) Let $R:\K^N\rightarrow\R$ be lower-semi-continuous for the Hausdorff convergence of convex bodies. Let $D\in \K^N$ and $0<V_0<|D|$. Then there exists a minimizer to the problem
\[\inf\left\{P(K)+R(K),\ K\in \K^N,\ K\subset D,\ |K|=V_0\right\}\]
(ii) Let $V_0>0$. Let $n\in \N^*$ and $F:(\R^+)^{2n+2}\rightarrow\R$ be coercive (meaning $\lim_{|x|\to\infty}F(x)=+\infty$) and lower-semi-continuous, and set 
\[R(K):=F(|K|, \tau(K), \lambda_{1}(K),\ldots,\lambda_{n}(K),\mu_{1}(K),\ldots,\mu_{n}(K))\] Then there exists minimizers to the problems
\[\inf\left\{P(K)+R(K),\ K\in \K^N\right\}\]
\[\inf\left\{P(K)+R(K),\ K\in \K^N, \ |K|=V_0\right\}\]
\end{theorem}

\begin{proof}
In both cases existence is proved using the direct method: let $(K_i)$ be a minimizing sequence.

\textbf{(i)}
Since $K_i\subset D$ for each $n$, then thanks to the Blaschke selection theorem and \eqref{eq:equivHausVol} we can extract a subsequence (still denoted $(K_i)$) and a compact convex $K^*$ such that $K_i\rightarrow K^*$ for the Hausdorff distance and in volume. We can pass to the limit in $|K_i|=V_0$ to get $|K^*|=V_0>0$, so that $K^*$ has non-empty interior. We deduce that $\varliminf R(K_i)\geq R(K^*)$ thanks to the hypothesis made on $R$ and that $P(K_i)\rightarrow P(K^*)$ by continuity of the  perimeter for convex domains (see for instance \cite[Proposition 2.4.3, (ii)]{BB}), thus getting existence.

\textbf{(ii)} We start with existence for the first of the two problems. Thanks to John's ellipsoid Lemma, there exists $c_i\in \R^N$ and ellipsoids $E_i$ such that 
\[E_i\subset K_i\subset c_i+N(E_i-c_i)\]
We have by monotonicity of the perimeter for convex bodies $P(K_i)\geq P(E_i)$, while we also have $\text{diam}(K_i)\leq N\text{diam}(E_i)$. As a consequence, if we assume by contradiction that (up to subsequence) $\text{diam}(K_i)\rightarrow+\infty$, then we first deduce $\text{diam}(E_i)\rightarrow+\infty$ so that $P(E_i)\rightarrow+\infty$, whence $P(K_i)\rightarrow+\infty$. The function $F$ being coercice and lower-semi-continuous it is therefore bounded from below, and we get the contradiction $P(K_i)+R(K_i)\rightarrow+\infty$. Therefore $\text{diam}(K_i)$ is bounded and we can assume by translation invariance of $P$ and $R$ that there exists a compact set $D\subset\R^N$ such that $K_i\subset D$ for each $i$.
Thanks to the Blaschke selection theorem and \eqref{eq:equivHausVol} we can extract a subsequence (still denoted $(K_i)$) and a compact convex $K^*$ such that $K_i\rightarrow K^*$ for the Hausdorff distance and in volume. The case $|K^*|=0$ is excluded, since it would lead to $|E_i|\rightarrow0$ and then $+\infty\leftarrow\lambda_1(E_i)\leq N^2\lambda_1(K_i)$ by monotonicity of $\lambda_1$, which yields $R(K_i)\rightarrow+\infty$ by coercivity of $F$ hence the contradiction $P(K_i)+R(K_i)\rightarrow+\infty$. As a consequence $|K^*|>0$, which means that $K^*$ has non-empty interior, and in particular there exists $D'\in\K^N$ such that $D'\subset \text{Int}(K^*)$. Since $d_H(K_i,K^*)\rightarrow 0$ we know thanks to \eqref{eq:opensubH} that $K_i\in \K^N_{D',D}$ for large enough $i$. As a consequence, the $K_i$ are uniformly Lipschitz in the sense that they verify the $\eps$-cone condition for some $\eps$ independent of $i$ (see Definition \ref{th:defeps} and Remark \ref{th:rk_eps}). We thus have continuity $\lambda_k(K_i)\rightarrow\lambda_k(K^*)$ and $\tau(K_i)\rightarrow\tau(K^*)$ (see \cite[Theorem 2.3.18]{H06}) and $\mu_k(K_i)\rightarrow\mu_k(K^*)$ (see \cite[Theorem 2.3.25]{H06}). Recalling the lower-semi-continuity of $F$ we deduce $\varliminf R(K_i)\geq R(K^*)$, and by continuity of the  perimeter for convex domains we also have $P(K_i)\rightarrow P(K^*)$. This finishes the proof of existence for the first problem.

Existence for the second problem is shown with the same argument, by noticing that the volume constraint passes to the limit.

\end{proof}

 \subsubsection{Selected examples}\label{ssect:selected}

Before moving on to the proof of Theorem \ref{th:examples} (which is the object of sections \ref{sect:Dirichlet} and \ref{sect:Neumann}), we discuss here some specific examples where $R$ involves spectral functionals and for which we can prove existence without a box constraint (in the spirit of Theorem \ref{th:general_existence} (ii)). We will make use of Theorem \ref{th:examples} in this section.

We start by considering spectral problems with a perimeter constraint, which have been studied in the literature without the additionnal convexity constraint (see for instance \cite{BBH}, \cite{dPV}, \cite{Bog}, \cite{BO}). Namely, given $p_0>0$, we are interested in the minimization problems \begin{equation}\label{eq:perim_cons}\inf\{\lambda_n(K),\ K\in \K^N,\ P(K)=p_0\}\end{equation} In \cite{BBH}, the authors use convexity for proving existence as well as $C^{\infty}$ regularity and some qualitative properties of minimizers of $\lambda_2$ under perimeter constraint in two dimensions. In fact, although their problem is set without a convexity constraint, they are able to show that solutions are in fact convex, thus yielding a bit of regularity to start with. On the other hand, in dimension $N=3$ there are eigenvalues for which the expected solutions are not convex (see \cite[Figure 2]{BO}), so that the convexity constraint would thus be meaningful in the minimization.

In our case we can prove existence together with $C^{1,1}$ regularity of minimizers. This is the object of next result. 
\begin{prop}Let $n\in\N^*$, $N\geq2$ and $p_0>0$. Then there exists a solution to problem \eqref{eq:perim_cons} and any such solution is $C^{1,1}$. \end{prop}

\begin{proof}The proof is divided into proof of existence and proof of regularity. 

\noindent{\bf Existence:} We use the direct method. Let $(K_i)$ be a minimizing sequence for \eqref{eq:perim_cons}. By definition there exists $C>0$ such that $\lambda_n(K_i)\leq C$ for each $i$. Since $\lambda_1(K_i)\leq\lambda_n(K_i)\leq C$, we deduce using Faber-Krahn inequality \[C|K_i|^{2/N}\geq \lambda_1(K_i)|K_i|^{2/N}\geq\lambda_1(B)|B|^{2/N}=:C_N\] with $B$ the unit ball, giving \begin{equation}\label{eq:penal_lowvolume}|K_i|\geq \left(\frac{C_N}{C}\right)^{N/2}\end{equation} On the other hand the perimeters $P(K_i)$ are bounded from above (in fact $P(K_i)=p_0$), yielding that the $K_i$ are uniformly bounded (up to translation), using \eqref{eq:diamPV}. We therefore get existence by proceeding as in the proof of Theorem \ref{th:general_existence} (i): thanks to the Blaschke selection theorem and \eqref{eq:equivHausVol} we thus find a subsequence (still denoted $(K_i)$) converging to some compact convex set $K^*$ in the Hausdorff sense  and in volume. The lower bound on volumes \eqref{eq:penal_lowvolume} thus ensures that $|K^*|>0$, so that the convex $K^*$ has nonempty interior. Hence there exists $D'\in\K^N$ such that $D'\subset \text{Int}(K^*)$. Since $d_H(K_i,K^*)\rightarrow 0$ we know thanks to \eqref{eq:opensubH} that $K_i\in \K^N_{D',D}$ for large enough $i$. We deduce that $\lambda_n(K_i)\rightarrow \lambda_n(K^*)$ using that $\lambda_n$ satisfies \eqref{eq:hypR4} thanks to Theorem \ref{th:examples}, and that $P(K_i)\rightarrow P(K^*)$ by continuity of the perimeter for convex domains. This finishes the proof of the existence part.

\noindent {\bf Regularity:} Let $K^*$ be any minimizer for \eqref{eq:perim_cons}. Following \cite[Remark 3.6]{dPV} we can show that there exists $\mu>0$ such that $K^*$ minimizes \begin{equation}\label{eq:uncons}\inf\{\lambda_n(K)+\mu P(K),\ K\in \K^N\}\end{equation} As a consequence we can apply Corollary \ref{th:shapecor} to get that $K^*$ is $C^{1,1}$. \end{proof}

 We now move on to problems of the kind \eqref{eq:pbshapeconstvol} with a volume constraint and with $R$ of  spectral type. These type of problems are related to the study of Blaschke-Santalo diagrams, see \cite{FL21Bla} and the numerical results in \cite{F21Dia}.  Again, we can drop the box constraint and still get existence:

\begin{prop} Let $N\geq2$, $V_0>0$ and $n\in\N^*$. There exist minimizers to the problems 
\begin{equation}\label{eq:pbDirichlet} \inf\{P(K)+\lambda_n(K),\ K\in\K^N, |K|=V_0\}\end{equation}\begin{equation}\label{eq:pbNeumann}\inf\{P(K)\pm\mu_n(K),\ K\in \K^N,\ |K|=V_0\}\end{equation} and any minimizer is $C^{1,1}$.\end{prop}
\begin{proof} The regularity assertion is a consequence of Theorem \ref{th:intro}. Let us prove existence of a solution for the two family of problems:
\begin{enumerate}\item Existence is obtained by applying Theorem \ref{th:general_existence} (ii).

\item For the minimization of $P+\mu_n$ we can directly apply Theorem \ref{th:general_existence} (ii) to get existence. For the minimization of $P-\mu_{n}$, first note the inequalities 
\begin{equation}\label{eq:ineq_diam} \text{diam}(K)\leq C(N)\frac{P(K)^{N-1}}{|K|^{N-2}} \;\;\;\;\;\; \ \mu_n(K)\leq \frac{C_n(N)}{\text{diam}(K)^2}\end{equation} for any convex body $K$, for some constants $C(N)$ and $C_n(N)$ only depending on the indicated parameters (for the first, recall \eqref{eq:diamPV} and see for instance \cite[Proposition 2.1 (b)]{Ro} for the second). Let $(K_i)$ be some minimizing sequence for problem \eqref{eq:pbNeumann}. The sequence $(P(K_i)-\mu_n(K_i))$ being bounded from above by definition, we find $C>0$ such that \begin{equation}\nonumber P(K_i)\leq C+\mu_n(K_i)\leq C'(N)(1+ \text{diam}(K_i)^{-2})\end{equation} for some dimensional constant $C'(N)$, using the second inequality of \eqref{eq:ineq_diam}. Now, for fixed $i$ we either have  $\text{diam}(K_i)\geq 1$, in which case we deduce $P(K_i)\leq 2C'(N)$, or $\text{diam}(K_i)\leq 1$. Thanks to the first inequality of \eqref{eq:ineq_diam} this yields \begin{equation}\nonumber\text{diam}(K_i)\leq \max\left\{1,\frac{C(N)(2C'(N))^{N-1}}{V_0^{N-2}}\right\}\end{equation} Therefore, using the translation invariance of $P$ and $\mu_n$ we can find a compact set $D$ such that $K_i\subset D$ for each $i$. Recalling that $R(K):=\mu_n(K)$ verifies \eqref{eq:hypR4} thanks to Theorem \ref{th:examples}, the rest of the proof of existence is as in Theorem \ref{th:general_existence} (i).\end{enumerate}\end{proof}

\begin{remark}\begin{itemize}\item One can also wonder about the minimization \[\inf\{P(K)-\lambda_n(K),\ K\in\K^N,\ |K|=V_0\}\] In this case the problem is ill-posed, as the box constraint is needed to ensure existence. In fact, one can see that the infimum is $-\infty$, choosing the sequence of long thin rectangle $R_\eps:=[0,V_0\eps^{-1}]\times[0,\eps]\times[0,1]^{N-2}$ for which $P(R_\eps)\leq C_N\eps^{-1}$ for some dimensional constant $C_N>0$ while $\lambda_n(R_\eps)\sim V_0^{-2}\pi^2\eps^{-2}$. \item Thanks to the isoperimetric inequality and the Faber-Krahn inequality (respectively the Szego-Weinberger inequality), it is known that the unique solution up to translation to the minimization of $P+\lambda_1$ (respectively of $P-\mu_2$) is any ball $B$ of volume $V_0$. On the other hand, if $n\geq2$ (respectively $n\geq3$) the problem \eqref{eq:pbDirichlet} (respectively \eqref{eq:pbNeumann}) has $C^{1,1}$ solutions which are not analytically known.
\item Inspired by \cite{FL21Bla}, one could wonder about the regularity properties of solutions to
$$\min\{P(K), \;K\in\K^N, \;|K|=V_{0}, \lambda_{n}(K)=\ell_{0}\},\;\;\;\max\{\lambda_{n}(K), \;K\in\K^N, \;|K|=V_{0}, \;P(K)=p_{0}\},$$
where $p_{0}>0, \ell_{0}>0$. In \cite[Corollary 3.13]{FL21Bla} it is proven when $N=2$ and $n=1$ that these problems are equivalent (for suitable choices of $p_{0}$ and $\ell_{0}$) and that solutions are $C^{1,1}$. Nevertheless, we were not able to apply our regularity result to these cases, so the regularity of solutions of these problems remains open in other cases ($N\geq 3$ or $n\geq 2$), up to our knowledge.
\end{itemize}\end{remark}

\subsubsection{Torsional rigidity and Dirichlet eigenvalues} \label{sect:Dirichlet}
If $D'\subset D\subset \R^N$ are convex bodies, we still denote by $\K^N_{D',D}$ the set \begin{equation}\nonumber\K^N_{D',D}:=\{K\in \K^N, \ D'\subset K\subset D\}\end{equation} Let us state now the main result of this section, which basically restates Theorem \ref{th:examples} for $\tau$. Indeed as we will see below, the proof of Theorem \ref{th:examples} for $\lambda_{n}$ will be a consequence of the same result for $\tau$.

\begin{prop}\label{th:LipEf}  Let $N\geq2$ and $D'\subset D \subset \R^N$ be convex bodies. Then there exists $C=C(D',D)$ such that for any $K'\subset K$ lying in $\K^N_{D',D}$ \begin{equation} 0\leq \tau(K)-\tau(K')\leq C|K\setminus K'|\label{eq:LipEf}\end{equation}\end{prop}

The proof of Proposition \ref{th:LipEf} is based on two preliminary lemmas: for convex bodies $K'\subset K$,\begin{enumerate}\item we construct a ``change of variable'' operator $T_{K,K'}:W^{1,\infty}(\Om_K)\rightarrow W^{1,\infty}(\Om_{K'})$ whose norm is uniformly bounded in $\K^N_{D',D}$, and which is the identity on a large part of $K'$, see Lemma \ref{th:lemmaeig} (recall that $\Om_{K}$ denotes the interior of $K$). \item we show uniform $W^{1,\infty}$-estimates of the torsion function of $\Om_K$, see Lemma \ref{th:lemmaEstEgf}. \end{enumerate}

\begin{lemma}[Change of variables]\label{th:lemmaeig}There exists $C=C(D',D)>0$ such that for any $K,K' \in \K^N_{D',D}$ with $K'\subset K$, there exists a  bi-Lipschitz homeomorphism $\phi:=\phi_{K',K}:\R^N\rightarrow\R^N$ such that the operator $T:=T_{K,K'}$ defined by
\begin{align}\nonumber 
T_{K,K'}:L^1(\Om_K)&\to L^1(\Om_{K'})\\f&\mapsto f\circ\phi
\nonumber\end{align} satisfies the requirements:
\begin{itemize} 
\item There exists $K''\subset K'$ such that $|K'\setminus K''|\leq C|K\setminus K'|$ and 
$$Tf(x)=f(x)\textrm{ a.e. in }K'',\textrm{ for any }f\in L^1(\Om_K)$$
\item For all $f_1$ and $f_2$ respectively in $H_0^1(\Om_K)$ and $W^{1,\infty}(\Om_K)$, $Tf_1$ and $Tf_2$ belongs to $H_0^1(\Om_{K'})$ and $W^{1,\infty}(\Om_{K'})$ respectively, with furthermore
\begin{align*}
 \|Tf_1\|_{H^1_{0}(\Om_{K'})}&\leq C\|f_1\|_{H^1_{0}(\Om_K)}\\
\|Tf_2\|_{W^{1,\infty}(\Om_{K'})}&\leq C\|f_2\|_{W^{1,\infty}(\Om_K)}
\end{align*}
\end{itemize}\end{lemma}
Note that this result is similar to \cite[Theorem 4.23]{BL2} but for a different class of sets, namely $\K^N_{D',D}$: it is unclear whether \cite[Theorem 4.23]{BL2} implies Lemma \ref{th:lemmaeig}, so we prefered to make our own proof of this result.

Let us recall that any $K\in \K^N$ has its boundary $\partial K$ naturally parametrized as a graph over the sphere. More precisely, we can assume up to translating that $0$ is contained in $\Om_{K}$, and then set $\rho(x):=\sup\{\lambda\geq0, \lambda x\in K\}$ for any $x\in \partial B$, called the radial function of $K$. Then the set $K$ is globally parametrized by $\rho$: \begin{equation}\label{eq:gauge} K=\left\{\lambda x\rho(x), x\in \partial B,\ \lambda\in[0,1]\right\}\end{equation}It is classical that $\rho\in W^{1,\infty}(\partial B)$ and moreover one can estimate $\nabla_{\tau}\rho$ in terms of $\rho$ \begin{equation}\label{eq:Nabdist} \|\nabla_{\tau} \rho\|_{L^{\infty}(\partial B)}\leq \frac{\left(\sup\rho\right)^2}{\inf\rho}\end{equation} (see for instance the computations leading to (3.13) in \cite{F}).

\noindent{\bf Proof of Lemma \ref{th:lemmaeig}:} 
We will assume up to translating that $0\in \Int(D)$. 
Let $K,K'\in \K^N_{D',D}$ with $K'\subset K$. The proof consists in building a bi-Lipschitz change of variables $\phi:K'\rightarrow K$ which is the identity on a large part of $K'$, and such that 
\begin{align}\|\phi\|_{W^{1,\infty}(\Om_{K'})}, \|\phi^{-1}\|_{W^{1,\infty}(\Om_K)}\leq C \label{eq:bilipW1} \end{align} 
for some constant $C=C(D',D)>0$ independent of $K$ and $K'$.

Let $\rho$ and $\rho'$ denote respectively the radial functions of $K$ and $K'$. 
Let $\alpha$ be defined over $\partial B$ by the relation \begin{equation} \label{eq:defalpha} \rho'-\alpha=c(\rho-\rho')\end{equation} for some $c>0$ that will be chosen later. 
Then $\alpha\leq \rho'$ and we get the estimate
$$\alpha = ((c+1)\rho'-c\rho)\geq(c+1)\inr(D')-c\;\diam(D)$$
where $\inr(D')$ is the inradius of $D'$. For $c=c(D',D)=\frac{\inr(D')}{2[\diam(D)-\inr(D')]}$ we get  \begin{equation}\label{eq:lowalpha}\alpha\geq \frac{\inr(D')}{2}\end{equation} which is a lower bound independent of $K,K'\in \K^N_{D',D}$.

 If $u\in \R^N\setminus\{0\}$ we denote by $x_u=u/|u|$. Let $\phi$ be defined over  $\R^N$ by the formulae
 \begin{equation*}\forall u\in \R^N\setminus\{0\}, \
\phi(u):=\begin{cases} \phi_1(u):=x_u\left(\frac{c+1}{c}|u|-\frac{\alpha(x_u)}{c}\right), \text{ if } |u|\geq \alpha(x_u)  \\ \phi_2(u):=u, \text{ if }|u|<\alpha(x_u)\end{cases}\end{equation*}
and $\phi(0):=0$. Observe that the function $\phi$ is continuous and increasing along any normal direction $x\in \partial B$, and it verifies by construction that \begin{equation}\nonumber \phi(0)=0,\ \phi(u)=\rho(x_u)x_u \text{ if } |u|=\rho'(x_u)\end{equation} This ensures that $\phi$ is a bijection from $K'$ to $K$. 

Define $K''$ as the (non necessarily convex) set on which $\phi$ is the identity, {\it i.e.} \[K'':=\{ \lambda x\alpha(x), \ x\in \partial B, \ 0\leq \lambda\leq 1\}\]
The mapping $u\in \R^N\setminus K''\mapsto x_u$ having Lipschitz constant $ 2(\min_{\R^N\setminus K''}|u|)^{-1}=2(\min\alpha)^{-1}$, then \[u\in \R^N\setminus K''\mapsto x_u\alpha(x_u)\] has Lipschitz constant only depending on $\min\alpha$ and $\|\nabla_\tau\alpha\|_{L^{\infty}(\partial B)}$. From the definition \eqref{eq:defalpha} of $\alpha$ and recalling \eqref{eq:Nabdist}, we deduce that $\phi_1$ has Lipschitz constant $L$ only depending on $c$, $\inr(D')$ and $\diam(D)$, hence only on $\inr(D')$ and $\diam(D)$. Now, $\phi_2$ is Lipschitz over $\R^N$ (with Lipschitz constant $1$), so that we deduce that $\phi$ is globally Lipschitz over $\R^N$. Indeed: let $u_0\in K''$, $u_1\in \R^N\setminus K''$ and pick $t\in[0,1]$ such that the point $u_t:=(1-t)u_0+tu_1$ verifies $u_t\in \partial K''$ with $[u_t,u_1]\subset \R^N\setminus K''$. Since $\phi(u_t)=\phi_1(u_t)=\phi_2(u_t)$ we can write \begin{align}\nonumber|\phi(u_0)-\phi(u_1)|&=|\phi_2(u_0)-\phi_2(u_t)+\phi_1(u_t)-\phi_1(u_1)|\\\nonumber &\leq 1\times|u_0-u_t|+L|u_t-u_1|\\&\leq \max(1,L)|u_0-u_1|\nonumber\end{align} Hence $\phi$ is globally Lipschitz and its Lipschitz constant only depends on $D'$ and $D$. The same arguments can be applied to $\phi^{-1}$, thus getting \eqref{eq:bilipW1}.  Let us now show that the operator defined by \begin{equation} \forall f\in L^1_{\text{loc}}(\Om_K),  \ Tf:=f\circ\phi \label{eq:defT} \end{equation}
satisfies the expected requirements. For $p\in\{2,\infty\}$, any $f\in W^{1,p}(\Om_K)$ satifies $f\circ\phi\in W^{1,p}(\Om_{K'})$ and the weak derivatives of $f\circ\phi$ can be expressed with the classical formula for the derivative of a composition (see for instance \cite[Theorem 1.1.7]{Ma}); furthermore, if $f\in H^1_0(\Om_K)$ then $f\circ \phi$ verifies $f\circ \phi=0$ a.e. outside $\Om_{K'}$, giving that $Tf\in H^1_0(\Om_{K'})$ since $\Om_{K'}$ is Lipschitz  (see \cite[3.2.16]{HP}). Together with \eqref{eq:bilipW1} we deduce that $T$ satisfies the second requirement.

By construction $\phi(u)=u$ if $u\in K''$, so that it only remains to show \begin{equation} \label{eq:volcontrol} |K'\setminus K''|\leq C |K\setminus K'|\end{equation}with $C$ uniform in the class $\K^N_{D',D}$.

It is classical that $|K|=\frac{1}{N}\int_{\partial B}\rho^N d\H^{N-1}$ (see for instance \cite[(1.53)]{Sc}) and similarly for $K'$ and $K''$ with $\rho'$ and $\alpha$ in place of $\rho$, respectively. Therefore
\begin{equation} \label{eq:twoVol} |K'\setminus K''|=\frac{1}{N}\int_{\partial B}\left(\rho'^N-\alpha^N\right)d\H^{N-1}, \  |K\setminus K'|= \frac{1}{N}\int_{\partial B}\left(\rho^N-\rho'^N\right)d\H^{N-1}\end{equation}Using the identity \begin{equation}x^N-y^N=(x-y)\sum_{k=0}^{N-1}x^ky^{N-1-k}\nonumber\end{equation}we obtain: \begin{equation*} |K'\setminus K''|\leq N(\diam(D))^{N-1}\int_{\partial B}\left(\rho'-\alpha\right) d\H^{N-1}=cN(\diam(D))^{N-1}\int_{\partial B}\left(\rho-\rho'\right) d\H^{N-1}\nonumber\end{equation*} recalling \eqref{eq:defalpha}. Likewise we get \begin{equation}\nonumber |K\setminus K'|\geq \left(N\frac{\inr(D')^{N-1}}{2^{N-1}}\right)\int_{\partial B}\left(\rho-\rho'\right) d\H^{N-1} \end{equation} recalling \eqref{eq:lowalpha}. 
This proves \eqref{eq:volcontrol} for some $C=C(N,D,D')$ and completes the proof. \qed

Next lemma is a control of the torsion function $\| u_{\Om_K}\|_{W^{1,\infty}(\Om_K)}$ uniformly in the class $\K^N_{D',D}$:

\begin{lemma} \label{th:lemmaEstEgf} Let $D \in \K^N$. There exists $C=C(D)>0$ such that for all $K\in \K^N$ with $K\subset D$, then\begin{equation}\nonumber \| u_{K}\|_{W^{1,\infty}(\Om_K)}\leq C\end{equation}
where $u_{K}=u_{\Om_{K}}$ is the torsion function defined in \eqref{eq:torsion}.\end{lemma}

\begin{proof}
\begin{itemize} \item {\bf $L^{\infty}$ estimate of $u_K$:} 
We apply a standard maximum principle argument.  We note first that $u_K\in C^0(\overline{\Om})$ (see \cite[Theorem 6.13]{GT}). We choose $x_0\in\Om_K$ and let \begin{equation}\nonumber w(x):=\frac{1}{2N}\big(\text{diam}(\Om_K)^2-|x-x_0|^2\big)\end{equation} The construction of $w$ ensures\begin{equation}\nonumber\left\{\begin{array}{ccccc} -\Delta w&=&-\Delta u_K &\text{ in }& \Om_K \\[2mm] w&\geq &u_K &\text{ over } &\partial \Om_K\end{array}\right.\end{equation} 
The maximum principle then writes \begin{equation}\nonumber 0\leq u_K\leq w\leq \frac{\text{diam}(\Om_K)^2}{2N}\text{ in } \Om\end{equation} 
so that \begin{equation}\label{eq:estDLinfty} \|u_K\|_{L^{\infty}(\Om_K)}\leq \frac{\text{diam}(\Om_K)^2}{2N}\end{equation} \item {\bf $L^{\infty}$ estimate of $\nabla u_K$:}  This is obtained in \cite[Lemma 1]{CF} ; we reproduce the proof for sake of completeness. 

Since $\Om_K$ is convex the corresponding torsion function $u_K$ is $1/2$-concave (see \cite[Theorem 4.1]{Ke}), yielding also that the level sets $\{u_K> c\}$ are convex for any $c\in \R$. Let $x_0\in \Om_K$ and take a supporting hyperplane $H$ to the convex set $A:=\{u_K> u_K(x_0)\}$, which we can assume to be $H= \{x_N=0\}$ without loss of generality. The convexity of $A$ ensures that A is located on one side of $H$, say that $A\subset \{x_N\geq0\}$. As $x_0\in H$ we must also have $A\subset\{x_N\leq d\}$ where $d:=\text{diam}(\Om_K)$. Hence \begin{equation}A\subset  \{0\leq x_N\leq d\}\label{eq:barreerF}\end{equation} This construction provides a natural barrier to $u_K$ at $x_0$. In fact, denoting by $F:=\{0<x_N<d\}$, we let $w:F\rightarrow \R$  be defined for $x\in F$ by, \begin{equation}w(x)=\frac{1}{2}x_N(d-x_N)+u_K(x_0)\label{eq:expr_w}\end{equation} Then see that $w$ verifies \begin{equation} \begin{cases}\begin{aligned} -\Delta &w=1& &\text{ in } F\\ &w=u_K(x_0)& &\text{ over } \partial F\end{aligned}\end{cases}\nonumber\end{equation} Furthermore it holds \begin{align}\forall x\in F,\ |\nabla w(x)|=|\partial_Nw(x)|\leq \frac{d}{2}\label{eq:gradw}\end{align}We can now estimate the gradient of $u_K$ at $x_0$. Noting that $u_K\in C^0(\overline{\Om})$ (see \cite[Theorem 6.13]{GT}), we have thanks to \eqref{eq:barreerF} and \eqref{eq:expr_w} that $w\geq c\geq u_K$ over $\partial A\subset\{u_K\geq u_K(x_0)\}$. We can therefore apply the maximum principle in the open set $A$ to get \begin{align} \nonumber\sup_{x\in \Om}\frac{u_K(x)-u_K(x_0)}{|x-x_0|}&\leq \sup_{x\in A}\frac{u_K(x)-u_K(x_0)}{|x-x_0|} \\ & \leq \sup_{x\in A}\frac{w(x)-w(x_0)}{|x-x_0|} \nonumber\\&\leq \sup_{x\in F}\frac{w(x)-w(x_0)}{|x-x_0|}\nonumber\end{align} Using \eqref{eq:gradw} we finally obtain\begin{equation}\sup_{x\in \Om}\frac{u_K(x)-u_K(x_0)}{|x-x_0|}\leq  \frac{d}{2},\nonumber\end{equation} that is \begin{equation} \|\nabla u_K\|_{L^{\infty}(\Om)}=\sup_{x,y\in \Om}\frac{|u_K(x)-u_K(y)|}{|x-y|}\leq \frac{1}{2}\text{diam}(\Om)\nonumber\end{equation}

\end{itemize}\end{proof} 

\begin{remark}We did not use the convexity of $\Om_K$ for estimating $\|u_K\|_{L^{\infty}(\Om_K)}$ in terms of diam$(\Om_K)$, so that estimate \eqref{eq:estDLinfty} holds for any bounded open set $\Om$. Moreover, one can also obtain a finer estimate relying on a symmetrization argument due to G. Talenti: if $v$ is the solution to \begin{equation}\nonumber \begin{cases} -\Delta v=1 \text{ in } \Om_K^{\sharp} \\ v\in H^1_0(\Om_K^{\sharp})\end{cases}\end{equation} where $\Om_K^{\sharp}$ is the ball centered at the origin having the same volume than $\Om_K$, then \cite[Theorem 1 (iv)]{Ta} implies \begin{equation} u_K^{\sharp}\leq v \text{ in } \Om_K^{\sharp}\nonumber \end{equation} with $u_K^{\sharp}$ denoting the symmetric decreasing rearrangement of $u_K$. This provides a Volume-type control of $\|u_K\|_{L^{\infty}(\Om_K)}$: \begin{equation}\nonumber \|u_K\|_{L^{\infty}(\Om_K)}=\|u_K^{\sharp}\|_{L^{\infty}(\Om_K^{\sharp})}\leq \|v\|_{L^{\infty}(\Om_K^{\sharp})}\leq C(N)|\Om_K|^{2/N}\end{equation} with $C(N)$ a dimensional constant\footnote{This was pointed out to us by D. Bucur.}.  This estimate implies \eqref{eq:estDLinfty} up to a dimensional multiplicative constant, as $\Om_K$ can always be included in some cube of size $\text{diam}(\Om_K)$. It can reveal to be very convenient for controlling $\|u_K\|_{L^{\infty}(\Om_K)}$  when we can bound $|\Om_K|$ while having $\text{diam}(\Om_K)\longrightarrow+\infty$. \end{remark}

We are now in a position to prove Proposition \ref{th:LipEf}:

\noindent{\bf Proof of Proposition \ref{th:LipEf}:} Let $K'\subset K$ be convex bodies. Recall that $\Om_K$ and $\Om_{K'}$ denote the interiors of $K$ and $K'$ respectively. Let $T:\Om_K\rightarrow \Om_{K'}$ the change of variables given by Lemma \ref{th:lemmaeig}. Let $u_K:=u_{\Om_K}$ the torsion function, solution of \eqref{eq:torsion}. Using $Tu_K\in H^1_0(\Om_{K'})$ as a test function in the variational formulation of $\tau(K')$ we can write \begin{align}\nonumber 0\leq \tau(K)-\tau(K')
\leq \int_{K}\Big(|\nabla (Tu_K)|^2-|\nabla u_K|^2\Big)-2\int_{K}(Tu_K-u_K)\nonumber\end{align} Using the properties of $T$ given by Lemma \ref{th:lemmaeig} we get\begin{align}\nonumber 0\leq \tau(K)-\tau(K')&\leq  \int_{K\setminus K''}(|\nabla (Tu_K)|^2-|\nabla u_K|^2)-2\int_{K\setminus K''}(Tu_K-u_K)\nonumber\\ &\leq |K\setminus K''|\Big(\|\nabla(Tu_K)\|^2_{L^{\infty}(\Om_{K'})}+2\|Tu_K\|_{L^{\infty}(\Om_{K'})}+ 2\|u_K\|_{L^{\infty}(\Om_K)}\Big)\nonumber\\&\leq C \max\left(\|u_K\|_{L^{\infty}(\Om_K)},\|\nabla u_K\|_{L^{\infty}(\Om_K)}^{2}\right)|K\setminus K'|\nonumber\end{align} for some $C=C(D',D)$. Lemma \ref{th:lemmaEstEgf} then yields the result.
\qed

We will now show Theorem \ref{th:examples} for $\lambda_{n}$. To that end we will use the following result that was proved in \cite[Theorem 3.4]{B03Reg}:
\begin{theorem}\label{th:Eig_Ef}Let $\Om'\subset\Om$ be bounded open sets. For any $n\in \N^*$ it holds\begin{equation}\nonumber\left|\lambda_n(\Om)-\lambda_n(\Om')\right|\leq  C(n)\lambda_n(\Om)^{N/2+1}\lambda_n(\Om')|\tau(\Om)-\tau(\Om')|\end{equation}where $ C(n)= 2n^2e^{1/4\pi}$.
\end{theorem}

\noindent{\bf Proof of Theorem \ref{th:examples} for $\lambda_{n}$:} Theorem \ref{th:Eig_Ef} gives that  for any $K\in \K^N_{D',D}$\begin{equation*}\left|\lambda_n(K)-\lambda_n(K')\right|\leq  C(n)\lambda_n(K)^{N/2+1}\lambda_n(K')| \tau(K)-\tau(K')|\nonumber
\end{equation*}
From monotonicity of the Dirichlet eigenvalues we have $\lambda_{n}(K)^{N/2+1}\lambda_{n}(K')\leq \lambda_{n}(D')^{2+N/2}$, and therefore the estimate from Proposition \ref{th:LipEf} gives the result. \qed

\subsubsection{Neumann eigenvalues}\label{sect:Neumann}
The purpose of this section is to prove Theorem \ref{th:examples} in the case of Neumann eigenvalues. We actually get a slightly better result (with no assumption of inclusion between $K$ and $K'$, see also Remark \ref{rk:improvment}), more precisely:

\begin{prop} \label{th:LipeigN} Let $N\geq2$. Let $D'\subset D \subset \R^N$ be convex bodies. For any $n\in \N^*$ there exists $C_n=C_n(D',D)>0$ such that for each $K, K'\in\K^N_{D',D}$\begin{equation}|\mu_n(K)- \mu_n(K')|\leq C_n|K\Delta K'|\label{eq:LipeigN_final}\end{equation}\end{prop}

\begin{remark}
This result is stronger than \cite[Theorem 4.2]{Ro}, proven in the same class but for a different distance between sets (namely, the $L^\infty$ distance between the radial functions). More precisely, it is proven in \cite{Ro} that there exists $C_n=C_n(D',D)$ such that for all $K', K\in\K^N_{D',D}$ with respective radial functions $\rho_K$ and $\rho_{K'}$ \[|\mu_n(K)- \mu_n(K')|\leq C_n\|\rho_K-\rho_{K'}\|_{\infty}.\] This latter result is not enough to provide \eqref{eq:LipeigN_final}, as one sees for example by taking $K=[0,1]^3$ and $K_i=K\cap\{x\in\R^3, x_3>1/i-x_{1}\}$ ($K_i$ is built by cutting the neighborhood of an edge). It is not hard to see that $|K\setminus K_i|\leq C\|\rho_{K}-\rho_{K_{i}}\|_{\infty}^2$ (we fixed an origin inside $K$, for example $(\frac12,\frac12,\frac12)$), thus contradicting the possibility of controlling $\|\rho_K-\rho_{K'}\|_{\infty}$ by $|K\Delta K'|$. On the other hand this result is implied by \eqref{eq:LipeigN_final}, recalling the expression of volume in terms of radial functions (see for instance \cite[(1.53)]{Sc}) \[|K\Delta K'|=\frac{1}{N}\int_{\partial B}|\rho_K^N-\rho_{K'}^N|\leq C\|\rho_K-\rho_{K'}\|_{\infty}\] with $C$ a constant only depending on $D'$ and $D$.
\end{remark}

As in the Dirichlet case we follow the general strategy of \cite{BL1,BL2}, as we were not able to apply their result, namely \cite[Theorem 6.11]{BL1}. The proof of Proposition \ref{th:LipeigN} relies on the two independent steps: \begin{enumerate}\item we construct an extension operator $\Pi_K:W^{1,\infty}(\Om_K)\rightarrow W^{1,\infty}(\R^N)$ whose norm uniformly bounded in $\K^N_{D',D}$.\item we provide $W^{1,\infty}$-estimates of Neumann eigenfunctions.\end{enumerate}
Note that  unlike in the Dirichlet case, we cannot rely on a statement like Theorem \ref{th:Eig_Ef} and we have to directly work with the variational formulation of eigenvalues (we could actually apply the same strategy for proving Theorem \ref{th:examples} for $\lambda_{n}$, though we thought it was more elegant to use Theorem \ref{th:Eig_Ef}).

The following lemma deals with the first item of this strategy:

\begin{lemma}[Extension operator]\label{th:extN}Let $N\geq2$ and $D'\subset D\subset \R^N$ be convex bodies. There exists $C=C(D',D)>0$ such that for any $K \in \K^N_{D',D}$ there exists a bounded operator \begin{equation}\nonumber \Pi_K:L^1(\Om_K)\rightarrow L^1(\R^N)\end{equation} satisfying the requirements:\begin{itemize} \item for any $f\in L^1(\Om_K)$, $\Pi_Kf(x)=f(x)$ for $a.e.\;x$ in $\Om_{K}$,  \item if $f\in W^{1,\infty}(\Om_K)$ then $\Pi_Kf\in W^{1,\infty}(\R^N)$ with \begin{equation} \nonumber \|\Pi_Kf\|_{W^{1,\infty}(\R^N)}\leq C\|f\|_{W^{1,\infty}(\Om_K)}\end{equation}\end{itemize}\end{lemma}
\begin{remark}This result could be seen as a consequence of \cite[Theorem II.1]{Ch2} which asserts the same result in the wider class of sets satisfying the $\eps-$cone condition (see Definition \ref{th:defeps}). Nevertheless, using that the domains we consider are convex, we are able to give a shorter proof of this result.\end{remark}

\begin{proof} The ideas are similar from the ones in the proof of Lemma \ref{th:lemmaeig}. We again assume up to translating that $0\in \Int(D')$, and let $\rho$ be the radial function associated to $K$. 

If $u\in \R^N\setminus\{0\}$ we set $x_u:=u/|u|$. We let \begin{equation} \widetilde{K}:=\left\{\lambda x(\rho(x)+1),\ x\in B,\ \lambda\in[0,1]\right\} \nonumber\end{equation}\begin{equation}\forall u\in \R^N,\ \phi_1(u):= \begin{cases} u &\text{ if } u\in K\\ \rho(x_u)x_u &\text{  if }u\notin K \end{cases}\nonumber\end{equation} and \begin{equation}\forall u\in \R^N,\ \phi_2(u):= \begin{cases} 1 &\text{ if } u\in K\\ \rho(x_u)-|u|+1 &\text{ if } |u|\in[\rho(x_u),\rho(x_u)+1]\\ 0 & \text{ if }|u|>\rho(x_{u})+1\end{cases}\nonumber\end{equation}  Functions $\phi_1$ and $\phi_2$ are built by gluing continuously Lipschitz functions; as in the proof of Lemma \ref{th:lemmaeig} we thus deduce that $\phi_1$ and $\phi_2$ are Lipschitz with Lipschitz constant only depending on $\|\nabla_{\tau}\rho\|_{L^{\infty}(B)}$ and $\min\rho$, hence only on $D'$ and $D$. Therefore there exists $C(D',D)$ such that \begin{equation}\label{eq:Neum_Lip}\begin{cases} \|\phi_2\|_{L^{\infty}(\R^N)}\leq1,\  \|\nabla\phi_2\|_{L^{\infty}(\R^N)}\leq  C(D',D)\\ \|D \phi_1\|_{L^{\infty}(\R^N)}\leq  C(D',D)\end{cases}\end{equation} We finally let, for any $f\in L^1(\Om_K)$ and $u\in\R^N$\begin{equation} \nonumber \Pi_Kf(u):=\begin{cases} f(\phi_1(u))\phi_2(u) &\text{ if } u\in \widetilde{K}\\ 0 & \text{ if }u\notin \widetilde{K}\end{cases}\end{equation} By construction $\Pi_Kf(u)=f(u)$ for $u\in \Om_K$, and $\|\Pi_Kf\|_{L^{\infty}(\R^N)}\leq \|f\|_{L^{\infty}(\Om_K)}$.  Since $\phi_1$ and $\phi_2$ are Lipschitz we have that $\Pi_Kf\in W^{1,\infty}(\R^N)$ if $f\in W^{1,\infty}(\Om_K)$ and further $\nabla \Pi_Kf=(f\circ\phi_1)\nabla \phi_2+\nabla (f\circ\phi_1)\phi_2$ a.e..  Using \eqref{eq:Neum_Lip} we deduce \begin{align}\nonumber \|\nabla \Pi_Kf\|_{L^{\infty}(\R^N)}= \|\nabla \Pi_Kf\|_{L^{\infty}(\widetilde{K})}&\leq \|\phi_2\nabla (f\circ\phi_1)\|_{L^{\infty}(\widetilde{K})}+\|(f\circ\phi_1)\nabla \phi_2\|_{L^{\infty}(\widetilde{K})}\\ &\leq \|\phi_2\|_{W^{1,\infty}(\R^N)}\big(\|\nabla(f\circ\phi_1)\|_{L^{\infty}(\widetilde{K})}+\|f\circ\phi_1\|_{L^{\infty}(\widetilde{K})}\big)\nonumber \\ &\leq  (1+C(D',D))^2\|f\|_{W^{1,\infty}(\Om_K)}\nonumber\end{align} This completes the proof of the lemma.\end{proof}

We now state a $W^{1,\infty}$-estimate  for Neumann eigenfunctions:

\begin{lemma} \label{th:lemmaEstEgN} Let $N\geq2$, $D'\subset D \subset \R^N$ be convex bodies, and $n\in\N^*$. There exists $C_n=C_n(N,D',D)>0$ such that for all $K\in \K_{D',D}^N$, \begin{equation}\| v_{K,n}\|_{W^{1,\infty}(\Om_K)}\leq C_n\nonumber\end{equation}
where $v_{K,n}$ is any Neumann eigenfunction associated to $\mu_n(K)$ and such that $\|v_{K,n}\|_{L^2(\Om_K)}=1$. 
\end{lemma}

\begin{proof} We fix $n\in \N$ and denote more simply $v_n:=v_{K,n}$. \begin{itemize}\item {\bf $L^{\infty}$ estimate of $v_n$:}  By \cite[Proposition 3.1]{Ro} (and the remark following), it holds \begin{equation}\nonumber \|v_n\|_{L^{\infty}(\Om_K)}\leq C_1\big((1+\sqrt{\mu_k(K)})C_2\big)^r\end{equation} where \begin{equation}\begin{cases} C_1=C_1(N)\\C_2=C_2(D',D)\\ r=r(N)\end{cases}\nonumber\end{equation} Since $\mu_n(K)\leq \lambda_n(K)\leq \lambda_n(D')$, we get the estimate \begin{equation}\|v_n\|_{L^{\infty}(\Om_K)}\leq C_n(N,D',D)\label{eq:LinftyN}\end{equation}\item {\bf $L^{\infty}$ estimate of $\nabla v_n$:} It is proved in \cite{Ma2} in any dimension $N\geq2$ that \begin{equation}\nonumber\|\nabla v_n\|_{L^{\infty}(\Om_K)}\leq C(N)|K|^{1/N}\mu_n(K)C_{\Om_K}^{-1}\|v_n\|_{L^{\infty}(\Om_K)}\end{equation} where $C_{\Om_K}$ is the isoperimetric constant relative to $\Om_K$, {\it i.e.} if $\Om\subset\R^N$ is a bounded open set\begin{equation} C_{\Om}:=\inf_{\substack{E\subset \Om \\ 0<|E|\leq |\Om|/2}}\frac{P(E,\Om)}{|E|^{1-1/N}}=\inf_{\substack{E\subset \Om\\ 0<|E|< |\Om|}}\frac{P(E,\Om)}{\min\{|E|,|\Om\setminus E|\}^{\frac{N-1}N}}\nonumber\end{equation}with $P(\cdot,\Om)$ denoting the relative perimeter in $\Om$. Together with \eqref{eq:LinftyN} this provides \begin{equation}\label{eq:LinftynabN}\|\nabla v_n\|_{L^{\infty}(\Om_K)}\leq C_n(N,D',D)C_{\Om_K}^{-1}\end{equation} Now it is shown in \cite[Corollary 2]{Th} that $C_{\Om}\geq\delta(\eps,b,N)>0$ for any Lipschitz domain $\Om$ with diameter $\leq b$ and satisfying the $\eps$-cone condition \eqref{eq:epscone}. As there exists $\eps=\eps(D',D)$ such that \eqref{eq:epscone} is satisfied for any $K\in \K^N_{D',D}$ (see Remark \ref{th:rk_eps}) we can conclude from \eqref{eq:LinftynabN} that \begin{equation}\|\nabla v_n\|_{L^{\infty}(\Om_K)}\leq C_n(N,D',D)\nonumber\end{equation}
\end{itemize} \end{proof}

We are now in a position to prove Proposition \ref{th:LipeigN}. The following is a combination of the proofs of Theorems 3.2 and 4.20 of \cite{BL1} adapted to the particular case of the Neumann Laplace operator, which we reproduce for the convenience of the reader.\\ 

{\noindent\bf Proof of Proposition \ref{th:LipeigN}:} 
We denote by $v_1,...,v_n$ $n$ first eigenfunctions of the Neumann Laplace operator on $\Om_K$ normalized for the $L^2$ norm ({\it i.e.} $\|v_k\|_{L^2(\Om_K)}=1$).  Recall that functions $v_k$ are orthogonal for the $L^2$ scalar product. Let $f:=\sum_{k=1}^n\alpha_kv_k$ with $\|f\|_{L^2(\Om_K)}=1$, which means $\sum_{k=1}^n\alpha_k^2=1$. 
Set 
\begin{equation}T:=R_{K'}\circ \Pi_K:L^1(\Om_K)\rightarrow L^1(\Om_{K'})\nonumber\end{equation}
 where $\Pi_K$ is the extension operator given by Lemma \ref{th:extN} and $R_{K'}$ is the restriction onto $\Om_{K'}$. Note first that it suffices to prove estimate \eqref{eq:LipeigN_final} under the additionnal condition $|K\Delta K'|\leq \eps_n$ for some small $\eps_n(D',D)>0$ depending on $n$, $D'$, $D$; indeed, using $\mu_n(K)\leq \lambda_n(K)\leq \lambda_n(D')$ for any $K\in \K^N_{D',D}$, we have \begin{equation}\nonumber |\mu_n(K)-\mu_n(K')|\leq2\lambda_n(D')\leq \frac{2\lambda_n(D')}{\eps_n}|K\Delta K'|\end{equation} if $|K\Delta K'|\geq\eps_n$.\\\\ {\bf Estimate from below of $\|Tf\|_{L^2(\Om_{K'})}$:} We will first prove \begin{equation}\|Tf\|_{L^2(\Om_{K'})}^2\geq1-C_n|\Om_K\setminus\Om_{K'}|\label{eq:LipeigN_L2f}\end{equation} for some $C_n=C_n(N,D',D)>0$, which immediately provides \begin{equation}\|Tf\|_{L^2(\Om_{K'})}^{-2}\leq 1+2C_n|\Om_K\setminus\Om_{K'}|\label{eq:LipeigN_below}\end{equation} whenever $|\Om_K\setminus\Om_{K'}|\leq1/2C_n$, using the inequality $(1-x)^{-1}\leq 1+2x$ if $0\leq x\leq1/2$.
 
  We have, as $R_{K'}\circ \Pi_{K}(f)=f$ on $\Om_{K}\cap \Om_{K'}$,
   \begin{align}\|Tf\|^2_{L^2(\Om_{K'})}\geq\|Tf\|^2_{L^2(\Om_K\cap\Om_{K'})}&=\|f\|^2_{L^2(\Om_K\cap\Om_{K'})}= \|f\|^2_{L^2(\Om_K)}-\|f\|^2_{L^2(\Om_K\setminus\Om_{K'})} \label{eq:LipeigN_L2f1}\end{align} But thanks to Lemma \ref{th:lemmaEstEgN} and recalling that $\sum_{k=1}^n\alpha_k^2=1$\begin{align}\nonumber\|f\|^2_{L^2(\Om_K\setminus\Om_{K'})}&\leq\Big(\sum_{k=1}^n\alpha_k\|v_k\|_{L^2(\Om_K\setminus\Om_{K'})}\Big)^2\leq \sum_{k=1}^n\|v_k\|^2_{L^2(\Om_K\setminus\Om_{K'})}\\&\leq \Big(\sum_{k=1}^n\|v_k\|^2_{L^{\infty}(\Om_K)}\Big) |\Om_K\setminus\Om_{K'}| \leq C_n(N,D',D) |\Om_K\setminus\Om_{K'}|\nonumber\end{align}
   Pluging this into \eqref{eq:LipeigN_L2f1} yields \eqref{eq:LipeigN_L2f}.\\
  
  \noindent {\bf Estimate from above of $\|\nabla Tf\|_{L^2(\Om_{K'})}$:} We now prove \begin{equation}\label{eq:LipeigN_above}  \|\nabla Tf\|^2_{L^2(\Om_{K'})}\leq \mu_n(K)+C_n|\Om_{K'}\setminus\Om_K|\end{equation}for some $C_n=C_n(N,D',D)$.
  
  Note that\begin{align}\nonumber \|\nabla Tf\|^2_{L^2(\Om_{K'})}&= \|\nabla Tf\|^2_{L^2(\Om_{K'}\cap\Om_K)}+ \|\nabla Tf\|^2_{L^2(\Om_{K'}\setminus\Om_K)}\\&= \|\nabla f\|^2_{L^2(\Om_{K'}\cap\Om_K)}+\|\nabla Tf\|^2_{L^2(\Om_{K'}\setminus\Om_K)}\nonumber\\&\leq  \|\nabla f\|^2_{L^2(\Om_K)}+ \|\nabla Tf\|^2_{L^2(\Om_{K'}\setminus\Om_K)}\label{eq:LipeigN_HTf}\end{align}  The Neumann eigenfunctions being orthogonal for the $L^2$ scalar product, we also have $\int_{\Om_K}\nabla v_k\cdot\nabla v_{k'}=\mu_k(K)\int_{\Om_K} v_kv_{k'}=0$ for $k\ne k'$. Furthemore $\|\nabla v_k\|_{L^2(\Om_K)}^2=\mu_k(K)\leq\mu_n(K)$ and we thus get \begin{equation} \label{eq:LipeigN_Hf} \|\nabla f\|^2_{L^2(\Om_K)}=\sum_{k=1}^n\alpha_k^2\|\nabla v_k\|^2_{L^2(\Om_K)}\leq \mu_n(K)\end{equation} On the other hand, denoting by $ C_{1}=C_{1}(D',D)$ and $C_{2}=C_{2}(N,D',D)$ the constants respectively given by Lemmas \ref{th:extN} and \ref{th:lemmaEstEgN}, we obtain \begin{align}\nonumber \|\nabla Tf\|^2_{L^2(\Om_{K'}\setminus\Om_K)}&\leq \|\nabla Tf\|^2_{L^{\infty}(\Om_{K'})}|\Om_{K'}\setminus\Om_K|\nonumber\\&\leq C_{1}\|f\|^2_{W^{1,\infty}(\Om_K)}|\Om_{K'}\setminus\Om_K|\nonumber\\&\leq C_{1}\Big(\sum_{k=1}^n\|v_k\|^2_{W^{1,\infty}(\Om_K)}\Big)|\Om_{K'}\setminus\Om_K|\nonumber\\&\leq C_{1}nC_{2}^2|\Om_{K'}\setminus\Om_K|\label{eq:LipeigN_Hf2}\end{align} With estimates \eqref{eq:LipeigN_Hf} and \eqref{eq:LipeigN_Hf2}, \eqref{eq:LipeigN_HTf} gives \eqref{eq:LipeigN_above} for $C_{n}(N,D',D)=C_{1}nC_{2}^2$.\\
  
  \noindent {\bf Min-max principle and conclusion:} Let us remind the following min-max principle: \begin{equation} \mu_n(K')=\min_{\text{dim}(V)=n}\max_{\substack{g\in V\\g\ne0}}\frac{\|\nabla g\|_{L^2(\Om_{K'})}^2}{\|g\|_{L^2(\Om_{K'})}^2}\label{eq:LipeigN_infsup}\end{equation} where the minimum is taken over all $n$-dimensional subspaces $V\subset H^1(\Om_{K'})$. 

In the view of \eqref{eq:LipeigN_L2f}, if $|\Om_K\setminus\Om_{K'}|\leq\eps_n$ for some $\eps_n=\eps_n(N,D',D)$, we deduce $\|Tf\|_{L^2(\Om_{K'})}>0$ for any  $f=\sum_{k=1}^n\alpha_k v_k$ with $\|f\|_{L^2(\Om_K)}=1$, thus getting that $Tv_1,\ldots,Tv_n$ are linearly independent as $(v_{1},\ldots,v_{n})$ also are. Setting $L_n:=\text{Vect}(v_1,...,v_n)$, formula \eqref{eq:LipeigN_infsup} therefore implies \begin{equation} \mu_n(K')\leq \max_{\substack{ g\in TL_n\\ g\ne0}}\frac{\|\nabla g\|_{L^2(\Om_{K'})}^2}{\|g\|_{L^2(\Om_{K'})}^2}= \max_{\substack{ f\in L_n\\ \|f\|_{L^2(\Om_K)=1}}}\frac{\|\nabla Tf\|_{L^2(\Om_{K'})}^2}{\|Tf\|_{L^2(\Om_{K'})}^2}\nonumber\end{equation} Putting estimates \eqref{eq:LipeigN_below} and \eqref{eq:LipeigN_above} provides \begin{align}\nonumber \mu_n(K')&\leq \big(\mu_n(K)+C_n|\Om_{K'}\setminus\Om_K|\big)\big(1+2C_n|\Om_K\setminus\Om_{K'}|\big)\\ &\leq\mu_n(K)+\widetilde{C}_n|\Om_K\Delta\Om_{K'}|\nonumber \end{align}for some constant $\widetilde{C}_n(N,D',D)$, if $|\Om_K\setminus\Om_{K'}|\leq\eps_n$, where we used $\mu_n(K)\leq\lambda_n(K)\leq\lambda_n(D')$. Switching the roles played by $K$ and $K'$ we get \eqref{eq:LipeigN_final}. \qed

\subsection{Optimality of the $C^{1,1}$ regularity}\label{ssect:sharp}

We show in this short section that the H\"{o}lder regularity obtained in Theorem \ref{th:shape1} is optimal. Precisely we prove: 
\begin{prop}\label{th:C11optim} In $\R^2$, there exists a quasi-minimizer of the perimeter under convexity constraint which is $C^{1,1}$ but not $C^2$.\end{prop} 

Let us first introduce some notations. If $\Om$ is a measurable set we define its (scaling invariant) Fraenkel asymmetry $\alpha(\Om)$, which is scaling invariant: \begin{equation}\nonumber \alpha(\Om):=\inf\left\{\frac{|\Om\Delta B|}{|\Om|}, \ B\subset \R^2 \text{ a ball }, \ |B|=|\Om|\right\}.\end{equation} We also denote $D(\Om):=(P(\Om)-P(B))/P(B)$ the normalized isoperimetric deficit, where $B$ is any ball with same volume than $\Om$. We call {\it stadium} a set $\Om\subset \R^2$ which is obtained as the convex envelope of two disjoint disks of same radius.

The following result is proved in \cite{ANF} (see also \cite[Theorem 1.2]{CCH}, and the introduction of \cite{CL}).
\begin{theorem}\label{th:optimalconstant} It holds \begin{equation}\nonumber \inf_{K\in\K^2}\frac{D(K)}{\alpha(K)^2}\simeq0.405585>0\end{equation} and equality is achieved at a particular stadium $K^*$.\end{theorem}
\noindent{\bf Proof of Proposition \ref{th:C11optim}:} Let us note that any stadium is $C^{1,1}$ but not $C^2$, as the curvature jumps from value $0$ on a flat part to a positive value on a semi-circle. Call $c^*$ the value of the infimum above,  and set $V_0:=|K^*|$ and $B_{V_0}$ a ball of volume $V_0$. Then Theorem \ref{th:optimalconstant} implies \begin{equation} P(K)-P(B_{V_0})-c^*P(B_{V_0})\alpha(K)^2\geq0\nonumber \end{equation} for any planar convex body $K$ of volume $V_0$, with equality at $K^*$. In other words, $K^*$ minimizes the functional $P-P(B_{V_0})-c\alpha^2$ among planar convex sets of volume $V_0$, with $c:=P(B)c^*$. If one proves that $\alpha^2$ satisfies hypothesis  \eqref{eq:hypR4}, then we deduce that $K^*$ is a quasi-minimizer of the perimeter under convexity constraint, which concludes the proof. Let $(K,K')\in\K^2$, and with no loss of generality (as $\alpha$ is invariant with scaling) let us assume that $|K|=|K'|=1$. We denote $B$ an optimal ball in the definition of $\alpha(K)$: we have \begin{equation}\nonumber  \alpha^2(K')-\alpha^2(K)\leq 2(\alpha(K')-\alpha(K))\leq 2\left(|K'\Delta B|-|K\Delta B|\right)\leq 2|K\Delta K'|\end{equation}
where the last inequality is obtained by easily checking that $K'\Delta B\subset (K\Delta K')\cup(K\Delta B)$. Inverting the roles played by $K$ and $K'$ we deduce that $\alpha$ verifies \eqref{eq:hypR4}, thus completing the proof. \qed

\section{Appendix}

\subsection{Parametrization of convex bodies in cartesian graphs} 

Let us start by covering a few preliminaries about Lipschitz sets. The following definition, first given by D. Chenais in \cite{Ch}, is a very convenient way of considering "uniformly" Lipschitz sets. 

\begin{definition}\label{th:defeps} Let $\eps>0$. We say that an open set $\Om\subset\R^N$ satisfies the $\eps$-cone condition if for any $x\in \partial\Om$ there exists a unit vector $\xi_x$ such that\begin{equation}\forall y\in B_{\eps}(x)\cap\overline{\Om},\ C(y,\xi_x,\eps)\subset\Om\label{eq:epscone}\end{equation} where we set \begin{equation}\nonumber C(y,\xi_x,\eps):=\{z\in \R^N, \langle z-y,\xi_x\rangle> \cos(\eps)|z-y|, 0< |z-y|< \eps\}.\end{equation}\end{definition}

\begin{remark}\label{th:rk_eps}For any fixed $M\geq m>0$, it can be shown that any open convex set $\Om$ such that there exists $x\in\R^N$ with \begin{equation} \nonumber \ B_m(x)\subset \Om\subset B_M(x)\end{equation} verifies the $\eps$-cone condition for some $\eps=\eps(m,M)$ (see for instance \cite[Proposition 2.4.4]{HP}).\end{remark}

The following proposition shows that one can see a convex set as the graph of a Lipschitz function with specific additional properties that will be used in the proof of Theorem \ref{th:shape1} (see Figure \ref{fig:convex_graph} for an illustration).

\begin{prop}\label{th:appgraph}Let $K\in\K^N$. For any $\widehat{x_0}\in\partial K$, there exists \begin{itemize}\item A hyperplane $H\subset \R^N$ containing $\widehat{x_{0}}$, \item A unit vector $\xi\in\R^N$ normal to $H$, \end{itemize} such that, denoting by $(x,t)$ a point in $H\times\R\xi$ coordinates (and hence denoting $\widehat{x_0}:=(x_0,0)$), it holds \begin{enumerate} \item  The set $\Om:=\left\{x\in H, \ (x+\R\xi)\cap \text{Int}(K) \ne \emptyset\right \}$ is open, bounded and convex, and the function \begin{align}\nonumber u: \Om  &\to\R \\ x&\mapsto \min\{t\in\R,\ (x,t)\in K\}\nonumber\end{align} is well-defined and convex. Furthermore, if $\eps$ is such that $\text{Int}(K)$ satifies the $2\eps$-cone condition (see Definition \ref{th:defeps}), then $B^\eps:=B_{\eps\tan(\eps)}(x_0)\subset\Om$ and $u_{|B^\eps}$ is $\tan(\eps)^{-1}$-Lipschitz. \item It holds \begin{align}\nonumber\{(x,u(x)), \ x\in \Om\}&\subset \partial K \\ K\cap (\Om\times\R\xi)&\subset\{(x,t)\in\Om\times\R\xi,\ u(x)\leq t\}\nonumber\end{align}\item For any open set $\omega\Subset B^\eps$, there exists $c>0$ such that \begin{equation}\nonumber \{(x,t)\in \omega\times\R\xi,\ u(x)\leq t\leq u(x)+c\}\subset K\end{equation}Furthermore we can choose $c$ only depending on $d(\omega,\partial B^\eps)$ and $\eps$.\end{enumerate}\end{prop}

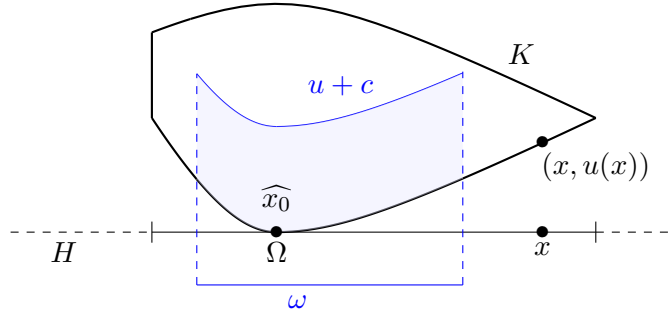
\begin{figure}
\centering
 \begin{tikzpicture}[scale=1.4]
\draw(-1.172,0) -- (3, 0);
\draw (-1.172,0.1) -- (-1.172,-0.1);
\draw(3,0.1) -- (3,-0.1);
\draw(0, 0)node[below]{$\Omega$};
\draw[thick,domain=-1.172:-0.75, variable=\x] plot ({\x}, {2*(sqrt(1+\x*\x)-1)});
\draw[thick,domain=-0.75:0, variable=\x, name path = h] plot ({\x}, {2*(sqrt(1+\x*\x)-1)});
\draw[thick,domain=0:1.75, variable=\x, name path= hh] plot ({\x}, {0.5*(sqrt(1+\x*\x)-1)});
\draw[thick,domain=1.75:3, variable=\x] plot ({\x}, {0.5*(sqrt(1+\x*\x)-1)});
\draw[thick,domain=-0.81:3, name path = g, variable=\x] plot ({\x}, {2*1.081-0.5*(sqrt(1+\x*\x)-1)});
\draw[thick,domain=-1.172:-0.81, name path = i, variable=\x] plot ({\x}, {2*1.081-0.5*(sqrt(1+\x*\x)-1)});
\draw[thick]  (-1.172,1.081) -- (-1.172,1.8917);

\draw[blue,domain=-0.75:0, name path=ii, variable=\x] plot ({\x}, {2*sqrt(1+\x*\x)-1});
\draw[blue,domain=0:1.75, name path=jj, variable=\x] plot ({\x}, {0.5*sqrt(1+\x*\x)+0.5});
\draw[blue] (0.6,1.2) node[above]{$u+c$};
\tikzfillbetween[of=ii and h]{blue!10, opacity=0.4}; 
\tikzfillbetween[of=jj and hh]{blue!10, opacity=0.4}; 

\draw (0,0) node{$\bullet$} ;
\draw (0,0.1) node[above]{$\widehat{x_0}$};
\draw (2.3,1.5) node[above]{$K$} ;
\draw (2.5,0) node[below]{$x$};
\draw (2.5,0) node{$\bullet$} ;
\draw (-2,0) node[below]{$H$};
\draw[dashed] (2.5,0) -- (2.5,0);
\draw[dashed] (-2.5,0) -- (-1.172,0);
\draw[dashed] (3,0) -- (3.7,0);
\draw[blue] (-0.75,-0.5) -- (1.75,-0.5);
\draw(2.5,0.85) node{$\bullet$};
\draw(3,0.85) node[below]{$(x,u(x))$};
\draw[blue] (0.2,-0.5) node[below] {$\omega$};
\draw[blue,dashed] (-0.75,-0.5) -- (-0.75,1.5);
\draw[blue,dashed] (1.75,-0.5) -- (1.75,1.52);
\end{tikzpicture}
\caption{\label{fig:convex_graph} Convex body in cartesian graph}
\end{figure}

\begin{proof}
The proof is inspired from \cite[Theorem 2.4.7]{HP}, with a few adaptations due to convexity. Since $K\in\K^N$, then $\text{Int}(K)$ satisfies the cone condition (see Definition \ref {th:defeps} and Remark \ref{th:rk_eps}). We can assume without loss of generality that it satisfies the $2\eps$-cone condition for some $\eps>0$ with $\tan(\eps)\leq1$. Let then $\xi:=\xi_{\widehat{x_0}}$ be a unit vector associated to $\widehat{x_0}$ and the $2\eps$-cone condition, that is \[\forall \widehat{x}\in K\cap B_{2\eps}(\widehat{x_0}),\ C(\widehat{x},\xi,2\eps)\subset  \text{Int}(K)\] We set $H:=\{\widehat{x}\in\R^N, \langle \widehat{x}-\widehat{x_0},\xi\rangle=0\}$. 

The function $u$ is well-defined by construction of $\Om$. The convexity of $K$ gives immediately that $u$ is convex: if $\lambda\in[0,1]$ and $x,y\in\Om$, the point $(1-\lambda)(x,u(x))+\lambda(y,u(y))\in K$ since $K$ is convex, giving that $u((1-\lambda)x+\lambda y)\leq (1-\lambda)u(x)+\lambda u(y)$ by definition of $u$.

For the rest of the proof we will write $(y,y_\xi)$ for the $H\times\R\xi$ coordinates of a point $\widehat{y}\in\R^N$. Any cone $C(\widehat{x},\xi,2\eps)$ can be written \begin{equation}\label{eq:cone_rewrite}C(\widehat{x},\xi,2\eps)=\left\{\widehat{y}\in B_{2\eps}(\widehat{x}),\ y_{\xi}-x_{\xi}> \frac{1}{\tan(\eps)}|y-x|\right\}\end{equation} Indeed, if $\widehat{y}\in\R^N$ and $\widehat{x}\in\R^N$ are such that $y_\xi-x_\xi\geq0$, then \begin{align}y_\xi-x_\xi>\cos(\eps)|\widehat{y}-\widehat{x}| &\Longleftrightarrow  (y_\xi-x_\xi)^2>\cos^2(\eps)\big(|y-x|^2+(y_\xi-x_\xi)^2\big)\nonumber\\&\Longleftrightarrow y_\xi-x_\xi>\frac{1}{\tan(\eps)}|y-x|\nonumber\end{align}Recalling that $B^\eps=\{x\in H,\ |x-x_0|< \eps\tan(\eps)\}$ we claim that \begin{equation}\label{eq:Appxeps}\forall x\in B^\eps, \ \begin{cases}(x,\eps)\in  \text{Int}(K)\\ (x,-\eps)\notin  \text{Int}(K)\end{cases}\end{equation}Indeed, let $x\in B^\eps$. Let us first show that $(x,\eps)\in  \text{Int}(K)$. Since $C(\widehat{x_0},\xi,2\eps)\subset \text{Int}(K)$ by the $2\eps$-cone condition, then it suffices to prove that $(x,\eps)\in C(\widehat{x_0},\xi,2\eps)$. But as $|x-x_0|< \eps$ it holds $(x,\eps)\in B_{2\eps}(\widehat{x_0})$, and furthermore $\tan(\eps)^{-1}|x-x_0|<\eps$, so that we deduce $(x,\eps)\in C(\widehat{x_0},\xi,2\eps)$ thanks to \eqref{eq:cone_rewrite} and hence $(x,\eps)\in \text{Int}(K)$. For the second assertion it is sufficient to prove that $C(\widehat{x_0},-\xi,2\eps)\subset \R^N\setminus \text{Int}(K)$, since in any case $(x,-\eps)\in C(\widehat{x_0},-\xi,2\eps)$ using again \eqref{eq:cone_rewrite}, as $(x,-\eps)\in B_{2\eps}(\widehat{x_0})$ with $\tan(\eps)^{-1}|x-x_0|<\eps=(-\eps)\times(-1)$. Suppose then by contradiction that there exists $\widehat{x}\in C(\widehat{x_0},-\xi,2\eps)\cap  \text{Int}(K)$; since $\widehat{x}\in B_{2\eps}(\widehat{x_0})$, it holds that $C(\widehat{x},\xi,2\eps)\subset  \text{Int}(K)$ by the $2\eps$-cone property. But then $\widehat{x_0}\in C(\widehat{x},\xi,2\eps)$, yielding $\widehat{x_0}\in \text{Int}(K)$, which is a contradiction. This finishes the proof of \eqref{eq:Appxeps}.

Thanks to \eqref{eq:Appxeps}, it holds $B^\eps\Subset \Om$. Let us show that $u_{|B^\eps}$ is Lipschitz continuous with Lipschitz constant $\tan(\eps)^{-1}$. If $x\in B^\eps$ then $(x,\eps)\in  \text{Int}(K)$ thanks to \eqref{eq:Appxeps}  and $(x,u(x))\in \partial K$ so that $[(x,u(x)),(x,\eps)]\subset K$. This ensures $-\eps<u(x)<\eps$ using again \eqref{eq:Appxeps}, so that in particular $(x,u(x))\in B_{2\eps}(\widehat{x_0})$. Let now $x,y\in B^\eps$. As $(x,u(x))\in B_{2\eps}(\widehat{x_0})$ we must have  $C((x,u(x)),\xi,2\eps)\subset  \text{Int}(K)$ by the $2\eps$-cone property; but then, as $(y,u(y))\in\partial K$ we get $(y,u(y))\notin C((x,u(x)),\xi,2\eps)$, giving \[u(y)-u(x)\leq \frac{1}{\tan(\eps)}|y-x|\] Reversing the roles played by $x$ and $y$, we deduce in fact that $u_{|B^\eps}$ is $\tan(\eps)^{-1}$-Lipschitz. This finishes the proof of the first requirement. 

The construction of $u$ ensures that the second requirement is verified. As for the third let $\delta$ be such that $d(\omega,\partial B^\eps)\geq\delta>0$. Set $c:=\delta\tan(\eps)^{-1}$ and let $x\in B^\eps$ and $u(x)\leq t\leq u(x)+c$. As $x\in B^\eps$, we have $(x,\eps)\in  \text{Int}(K)$ and $(x,u(x))\in \partial K$ so that $[(x,u(x)),(x,\eps)]\subset K$, hence it suffices to show that $t< \eps$ to get the claim. As it holds that $|x-x_0|< \eps\tan(\eps)-\delta$, then recalling that $u_{|B^\eps}$ is $\tan(\eps)^{-1}$-Lipschitz we get \begin{equation*}\nonumber t\leq u(x)+c< \tan(\eps)^{-1}\times (\eps\tan(\eps)-\delta)+c=\eps\nonumber\end{equation*} This finishes the proof of the third point and hence the proof of the Proposition.
\end{proof}

\subsection{Proof of Proposition \ref{prop:hausdorff}}

\begin{proof}Let $(K_n)$ be a sequence of convex bodies verifying $K_n\subset D$ where $D\in\K^N$.
\begin{itemize}\item Let us first focus on
\begin{equation*} d_H(K_n,K)\rightarrow0 \Longleftrightarrow |K_n\Delta K|\rightarrow0\end{equation*}
The direct sense is proved in \cite[Proposition 2.4.3, (ii)]{BB}. As for the converse, suppose $|K_{n}\Delta K|\to 0$, and assume by contradiction that $(K_n)$ does not converge to $K$ for $d_H$. Up to extracting we can therefore suppose that there exists $\eps>0$ such that \begin{equation}\label{eq:HausVol} \forall n \in\N,\  d_H(K_n,K)\geq\eps\end{equation} Thanks to the Blaschke selection theorem  which states that $\{L \text{ compact convex of }\R^N, \ L\subset D\}$ is compact for the Hausdorff distance, up to further extraction there exists $K_{\infty}$ compact convex such that $K_n\rightarrow K_{\infty}$ for $d_H$.  Using again \cite[Proposition 2.4.3, (ii)]{BB} and since $|K_n\Delta K|\rightarrow0$ we must have $K_{\infty}=K$, contradicting \eqref{eq:HausVol}.

\item  We now assume that $d_H(K_n,K)\rightarrow0$, and let $C\in\K^N$ be such that $C\subset \text{Int}(K)$. We want to prove
\begin{equation*} C\subset K_n\;\; \text{ for large }\; n.\end{equation*}
 There exists $\eps>0$ such that $d(C,\partial K)\geq\eps$. Since $K_n\rightarrow K$ for $d_H$, we also have $\partial K_n\rightarrow \partial K$ for $d_H$ (see \cite[Lemma 1.8.1]{Sc}). This gives $d(C,\partial K_n)\geq\eps/2$ for large $n$. Let us assume to get a contradiction that we do not have $C\subset K_n$ for $n$ large enough. Up to extraction we can therefore suppose that $C\cap (\R^N\setminus K_n)\ne \emptyset$ for each $n$. But then $C \subset \R^N\setminus K_n$ thanks to the convexity of $C$, since otherwise there would exist $x\in C\cap \partial K_n$ which is in contradiction with $d(C,\partial K_n)\geq\eps/2$. This rewrites $K_n\subset \R^N\setminus C$ for each $n$,  yielding $K\subset \overline{\big(\R^N\setminus C\big)}$ at the limit. This contradicts the hypothesis Int$(C)\subset K$. \end{itemize}\end{proof}

\noindent{\bf Acknowledgements: }
The authors would like to thank Dorin Bucur and Guillaume Carlier for very helpful discussions about this work. The authors also thank Guillaume Carlier for providing a version of \cite{CCL}. This work was partially supported by the project ANR-18-CE40-0013 SHAPO financed by the French Agence Nationale de la Recherche (ANR).

\bibliographystyle{plain}
\bibliography{./bib_shape}

\end{document}